\documentclass[11pt]{amsart}

\usepackage{amssymb}
\usepackage{latexsym}
\usepackage{amsmath}
\usepackage{amsthm}
\usepackage{amsfonts}
\usepackage{color}
\usepackage{graphicx}
\usepackage[left=3cm,top=4cm,right=3cm,bottom=4cm]{geometry}

\usepackage[normalem]{ ulem }

\newcommand{\R}{\mathbb{R}}
\newcommand{\N}{\mathbb{N}}
\newcommand{\Z}{\mathbb{Z}}

\newcommand{\E}{\mathcal{E}}
\newcommand{\D}{\mathcal D}
\newcommand{\T}{\mathbb{T}}
\newcommand{\e}{\epsilon}

\theoremstyle{plain}
\newtheorem{defi}{Definition}[section]

\newtheorem{teo}[defi]{Theorem}
\newtheorem{cor}[defi]{Corollary}
\newtheorem{lema}[defi]{Lemma}

\theoremstyle{remark}
\newtheorem{remark}[defi]{Remark}

\theoremstyle{definition}

\theoremstyle{remark}

\numberwithin{equation}{section}

\begin{document}

\title[]{Some results for the large time behavior of Hamilton-Jacobi Equations with Caputo Time Derivative}

\author{Olivier Ley}
\address{
Olivier Ley: Univ Rennes, INSA Rennes, CNRS, IRMAR - UMR 6625, F-35000 Rennes, France.
\newline {\tt olivier.ley@insa-rennes.fr}
}

\author{Erwin Topp}
\address{
Erwin Topp:
Departamento de Matem\'atica y C.C., Universidad de Santiago de Chile,
Casilla 307, Santiago, CHILE.
\newline {\tt erwin.topp@usach.cl}
}

\author{Miguel Yangari}
\address{Miguel Yangari: Departamento de Matem\'atica, Escuela Polit\'ecnica Nacional, Ladr\'on de Guevara E11-253, P.O. Box 17-01-2759, Quito, Ecuador.
\newline {\tt miguel.yangari@epn.edu.ec}
}
\date{\today}

\begin{abstract}
We obtain some H\"older regularity estimates for an Hamilton-Jacobi with fractional time derivative
of order $\alpha \in (0,1)$ cast by a Caputo derivative. The H\"older seminorms are independent of time,
which allows to investigate the large time behavior of the solutions. We focus on the Namah-Roquejoffre
setting whose typical example is the Eikonal equation.
Contrary to the classical time derivative case $\alpha=1$, the convergence of the solution on the so-called projected
Aubry set, which is an important step to catch the large time behavior,
is not straightforward. Indeed, a function with nonpositive Caputo derivative for all time
does not necessarily converge; we provide such a counterexample.
However, we establish partial results of convergence under some geometrical assumptions. 
\end{abstract}

% \keywords{Integro-Differential Equations, Regularity, Comparison Principles, Large Time Behavior, Strong Maximum Principles}
%
% \subjclass[2010]{35R09, 35B51, 35B65, 35D40, 35B10, 35B40}
\maketitle

%%%%%%%%%%%%%%%%%%%%%%%%%%%%%%%%%%%%%%%%%%%%%%%%%%%%%%%
\section{Introduction.}

In this note we are interested in nonlocal Hamilton-Jacobi equations with the form
\begin{align}\label{eq}
\partial_t^\alpha u + H(x,Du) = 0 \quad \mbox{in} \ Q := \T^N \times (0, +\infty).
\end{align}
subject to the initial condition
\begin{align}\label{u0}
u(\cdot, 0) = g \quad \mbox{in} \ \T^N,
\end{align}
for some $H\in C(\T^N\times \R^N)$ and $g \in \mathrm{Lip}(\T^N)$ given.

The nonlocal nature of the problem is cast by the operator $\partial_t^\alpha$, which denotes the Caputo time derivative of order $\alpha\in(0,1)$, starting at time zero. For $\phi \in C^1(0,+\infty)$ it is defined as
\begin{align}\label{Caputo}
\partial_t^\alpha \phi(t) = \frac{1}{\Gamma(1 - \alpha)} \int_{0}^{t} \frac{\phi'(s)}{|t - s|^\alpha} ds,
\end{align}
where $\Gamma$ is the Gamma function that acts as a normalizing constant making $\partial_t^\alpha$ become the usual first order derivative when $\alpha \to 1^-$,  see~\cite{diethelm10, gkmr14} and references therein.
Following the ideas of~\cite{acv16}, where under appropriate assumption on the function $\phi$, its Caputo derivative in~\eqref{Caputo} can be equivalently computed as
\begin{equation}\label{Caputo2}
\partial_t^\alpha \phi(t) = \tilde c_\alpha \int_{-\infty}^{t} \frac{\phi(t) - \phi(s)}{|t - s|^{1 + \alpha}} ds,
\end{equation}
for some $\tilde c_\alpha > 0$, and where we have extended $\phi$ as $\phi(t) = \phi(0)$ for $t < 0$.
This operator has a nonlocal degenerate elliptic nature in the sense of Barles and Imbert~\cite{bi08} that allows to conclude comparison principle among viscosity solutions as it is proved in~\cite{ty17}. This is a powerful reason to consider Caputo derivative instead of other fractional derivatives, as, for example, Riemann-Liouville derivative defined for adequate functions $\phi$ as
\begin{equation*}
\partial_{RL}^\alpha \phi(t) = c_\alpha \frac{d}{dt} \int_{0}^{t} \frac{\phi(z)}{(t - z)^\alpha}dz.
\end{equation*}
Moreover, as it can be seen in~\cite[Chapters 5,6]{diethelm10}, Caputo derivative is more adequate to deal with the most classical notion of initial condition compared with Riemann-Liouville problems, where the initial condition is understood in a generalized sense.

Coming back to~\eqref{eq}, and more specifically to the Hamiltonian $H$, we assume throughout this paper $H(x,p)$ is periodic in $x$ and coercive in $p$. 

We focus our attention into Bellman-type Hamiltonians with the classical structure related to optimal control problems with compact control set, that is, satisfying the regularity/growth condition
\begin{equation}\label{H}%\tag{{\bf H1}}
\left \{ \begin{array}{l}
\displaystyle |H(x,p) - H(y,p)| \leq c_H(1 + |p|)|x - y|, \\
%H(x,p) \geq c_H^{-1} |p| - c_H.
\displaystyle \mathop{\rm lim}_{|p|\to +\infty}  \mathop{\rm inf}_{x\in \T^N} H(x,p)=+\infty.
\end{array}
\right .
\end{equation}
for some $c_H > 0$, and for all $x,y \in \T^N, p \in \R^N$.

We are also interested in the case $H$ has superlinear growth in the gradient, common in control problems with unrestricted control space. We refer to this case through the following assumption: there exists $m > 1$ (the ``gradient growth") and $A > 0$, $c_H > 1$ such that
\begin{equation}\label{Hsuper}%\tag{{\bf H2}}
\left \{ \begin{array}{l} 
\mu H(x, \mu^{-1} p) - H(x, p) \geq (1 - \mu) \Big{(} c_H^{-1} |p|^{m} - A\Big{)},\\
|H(x,p) - H(y,p)| \leq c_H (1 + |p|^{m})|x - y|, 
\end{array} \right .
\end{equation}
for some $c_H>0$ and all $\mu \in (0,1)$, $x, y \in \T^N, \ p \in \R^N$.
We notice at once that the first condition in~\eqref{Hsuper}
implies that there exist $C > 0$ such that
\begin{equation}\label{coerc-Hsuper}
  H(x,p) \geq C^{-1}|p|^m - C \qquad\text{for all $x\in \T^N,p\in \R^N$,}
\end{equation}
and the same inequality holds with $m=1$ if the second condition holds in~\eqref{Hsuper}
and $H$ is convex.

%we assume that the Hamiltonian $H$
%is coercive in the gradient variable with sublinear growth, and Lipschitz continuous
%with respect to $x$, that is, there exists $c_H \geq 0$ such that for all   we have

Our interest is the analysis of the behavior of the solutions to~\eqref{eq}-\eqref{u0} and make a contrast with the classical case $\alpha = 1$, namely, the Hamilton-Jacobi equation
\begin{equation}\label{HJ1}
\partial_t u + H(x,Du) = 0 \quad \mbox{in} \ Q.
\end{equation}
There is a vast literature regarding problem~\eqref{HJ1}-\eqref{u0}. We refer the surveys of Barles and Ishii in~\cite{abil13} and references therein for the basics about this problem, like existence, uniqueness and regularity.

A natural question that arises is the analysis of the behavior of the solution for long times.
There too, there are a lot of references when $\alpha=1$,
see~\cite{fathi98, nr99, bs00, fs04, ds06, ishii09, bim13, abil13} and the references therein.
Here, we focus on the
nowadays classical framework of Namah and Roquejoffre paper~\cite{nr99},
where the authors address the long time behavior of~\eqref{HJ1} under the assumption
that 
\begin{equation}\label{NR}
\left \{
\begin{split}
& \text{$H(x,p) = F(x,p) - f(x)$ for all $(x,p)$, with $F$, $f$ continuous,} \\    
& \text{$F(x,\cdot)$ is convex for all $x\in \T^N$,}\\    
& F(x,p) > F(x,0)=0 \quad \mbox{for all} \ x \in \T^N, p \in \R^N\setminus\{0\},\\
& f\geq 0 \text{ in $\T^N$}.
\end{split}
\right .
\end{equation}
It is possible to generalize slightly the above assumptions but we choose
to state them in this form since all the main difficulties are present.

Under these assumptions,~\eqref{eq} reads
\begin{align}\label{eq-new}
\partial_t^\alpha u + F(x,Du) = f(x),
\end{align}
$f\in {\rm Lip}(\T^N)$ and
\begin{eqnarray}\label{zerof}
Z:=\{x\in\T^N: f(x)=\mathop{\rm min}_{\T^N}f\},
\end{eqnarray}
the so-called projected Aubry set \cite{fs04}, which is a compact subset of $\T^N$.
It follows from Lions, Papanicolaou and Varadhan~\cite{lpv86} 
that
the so-called \textsl{ergodic problem}
\begin{equation}\label{ergodic}
F(x, Dv) = f(x)+c \quad x \in \T^N,
\end{equation}
has a solution $(c,v)\in \R\times W^{1,\infty}(\T^N)$ and $c$ is unique.
Actually, under our assumptions, it is easy to see that
$c= -\mathop{\rm min}_{\T^N}f$.

%Without loss of generality, we assume that $\mathop{\rm min}_{\T^N}f=0$
%and set
%\begin{eqnarray}\label{zerof}
%Z:=\{x\in\T^N: f(x)=0\}.  
%\end{eqnarray}

It is then expected that the long time behavior of the solution $u$ of~\eqref{eq}-\eqref{u0}
is given by the asymptotic expansion
\begin{eqnarray}\label{asympt-exp}
u(x,t)+ ct^\alpha = v(x)+o(1), \quad \text{$o(1)\to 0$ uniformly as $t\to +\infty$},
\end{eqnarray}
where $(c,v)$ is a solution of~\eqref{ergodic}.

In the local case $\alpha=1$, this asymptotic behavior is established
in~\cite[Theorem 1]{nr99}. The proof relies
basically on three steps.
The first step is to obtain that the set $\{u(\cdot,t)+ct, t\geq 0\}$ is relatively
compact in $W^{1,\infty}(\T^N)$.
Let us point out that this property is not sufficient and
the main difficulty is to prove the full convergence of $u(\cdot, t)+ct$ as $t \to +\infty$. 
The second step is to notice
that $\partial_t (u(x,t)+ct)\leq 0$ for $x\in Z$, from which one infers
that $u(\cdot,t)+ct$ is nonincreasing in time
on $Z$, so it converges uniformly to a Lipschitz continuous
function $\phi$ on $Z$ as $t\to +\infty$.
The third step is to take the half-relaxed limits
\begin{eqnarray*}
  \overline{u}(x)=\mathop{\rm lim\,sup}_{y\to x,s\to t,\e\to 0}u(y,\frac{s}{\e})+c\frac{s}{\e}
  \quad \text{and}\quad
 \underline{u}(x)=\mathop{\rm lim\,inf}_{y\to x,s\to t,\e\to 0}u(y,\frac{s}{\e})+c\frac{s}{\e},
\end{eqnarray*}
which are, respectively, a sub- and a supersolution of~\eqref{ergodic}
and then to apply a strong comparison result for~\eqref{ergodic} with
the ``Dirichlet boundary condition'' $\overline{u}=\underline{u}=\phi$ on $Z$.
It follows $\overline{u}=\underline{u}$ in $\T^N$, which gives the
desired full convergence.

Our goal is to develop a similar procedure for the Caputo fractional case
$\alpha\in (0,1)$.
The first step holds.
Indeed, the elliptic properties shown by the expression~\eqref{Caputo2} are not only related to comparison principles but also to regularity in the time variable. Using Ishii-Lions method for nonlocal problems presented in~\cite{bci11}, we are able to prove that bounded solutions to~\eqref{eq} are $\alpha$-H\"older continuous in time, uniformly in $Q$, see Theorem~\ref{teoHoldert}.
We recall that such a property does not come from the contraction principle given by comparison arguments as in the classical case, basically because such a property is not known for fractional problems, where the influence of the "memory" put troubles in the analysis of problems shifted in time. Moreover, in the local case~\eqref{HJ1}, Lipschitz estimates in time allows to use Rademacher's Theorem to regard $u_t$ as an $L^\infty$ function, and extract the boundedness of $Du$ through the coercivity of $H$. Such a program cannot be carried out directly in~\eqref{eq} since $\alpha$-H\"older functions are not sufficient to make $\partial_t^\alpha u$ in $L^\infty$. Nevertheless, we can get equicontinuity in the space variable by a regularization procedure via inf and sup convolutions, see Theorem~\ref{holderx}.

Similarly, the third step is identical to the
one in the classical case since the ergodic problem~\eqref{ergodic}
is the same for all $\alpha\in (0,1]$.

It follows that the limiting step is the second one. We still have
$\partial_t^\alpha (u(x,t)+ct^\alpha)\leq 0$ for $x\in Z$ but it is not anymore
sufficient to infer neither that $u(\cdot,t)+ct^\alpha$ is nonincreasing in time on $Z$,
nor that it converges.
As we show in Section~\ref{seccounterexample},
it is possible to have bounded functions with signed $\alpha$-order Caputo derivative,
but  not converging as $t \to \infty$, which is a surprising result interesting by itself.

To overcome this difficulty and obtain the convergence on $Z$,
we need to give additional assumptions on the geometry of $Z$ and ${\rm argmin}\{g\}$.
Precise assumptions are stated in Section~\ref{secLTB} but basically, we assume
\begin{eqnarray}
  && \label{argminu0} Z\cap {\rm argmin}\{g\}\not= \emptyset,
  \\
  && \label{courbe-connect}\text{for any $z\in Z$, there exists a rectifiable curve in $Z$
      joining  ${\rm argmin}\{g\}$ and $z$.}
\end{eqnarray}
These assumptions are inspired from the classical case, where it is known that
the solution $u$ of~\eqref{HJ1}
is the value function of an optimal control problem for which
assumptions~\eqref{argminu0}-\eqref{courbe-connect} means roughly that we can
travel with minimal cost on $Z$.
But let us point out that these geometrical assumptions are not required in the
classical case and that there is no rigorous link between the fractional case
and the expected control problem, which could make these ideas rigorous.
See Camilli, De Maio, Iacomini~\cite{cdi18}
for the precise statement of a related control problem which should be associated
with the case $\alpha\in (0,1)$, and some discussion in this direction.

It follows that we need to translate these ideas in the PDE framework
building some suitable supersolutions, which tends to 0 thanks
to~\eqref{argminu0}-\eqref{courbe-connect}. By comparison, we obtain
estimates
\begin{eqnarray*}
&& \mathop{\rm min}_{\T^N}g \leq u(z,t)+ct^\alpha\leq \mathop{\rm min}_{\T^N}g+ {\rm Lip}(g)\, {\rm length}(\gamma)\, \E(t),
\quad \text{for any $z\in Z$,}  
\end{eqnarray*}
where $\gamma$ is a rectifiable curve on $Z$ joining ${\rm argmin}\{g\}$ and $z$,
and $\E$ is the solution of a fractional ODE with limit 0 at $+\infty$; see Theorem~\ref{cvusurZ1}
and Theorem~\ref{cvZ-eikonal} for an extension to some possibly infinite length curves.
This implies the convergence of $u$ on $Z$ from which we deduce easily~\eqref{asympt-exp},
see Corollary~\ref{ltb-caputo}.
This approach is new but we think it does not provide optimal results. In particular,
it relies too heavily on the geometry of $Z$, which is not the case in the
classical approach described above for $\alpha=1$. To go further, one
would need some quantitative estimates on how much nonpositive is 
$\partial_t^\alpha (u(x,t)+ct^\alpha)$, which seems difficult to obtain,
even in the case $\alpha=1$.

The paper is organized as follows. We start by introducing precisely
the Caputo fractional operator and recalling some useful properties.
Then we establish some regularity estimates for the solution of~\eqref{eq}
in Section~\ref{sec:reg}.
Section~\ref{seccounterexample} is devoted to a counterexample showing that a
bounded function with nonnegative Caputo derivative does not necessarily
converges. Finally, some positive results for the large time behavior
of the solution of~\eqref{eq} are proved in Section~\ref{secLTB}.
\medskip

\noindent{\bf Notations and preliminaries.}
We will write the Caputo derivative using~\eqref{Caputo2} (with $\tilde{c}_\alpha=1$ for simplicity).
More precisely, let $\phi: (0,+\infty)\to \R$. We extend $\phi$ to $(-\infty,0)$ by setting $\phi(s)=\phi(0)$ for $s<0$
and define, when it exists, for every $t> 0$ and $0<\delta <t$,
\begin{equation}\label{Caputo2-not}
  \partial_t^\alpha \phi(t) = \int_{-\infty}^{t} \frac{\phi(t) - \phi(s)}{|t - s|^{1 + \alpha}} ds
  = \partial_t^\alpha[t-\delta]\phi(t)+ \partial_t^\alpha[t-\delta,t]\phi(t),
\end{equation}
where, for $a<b\leq t$, we set
\begin{eqnarray*}
\partial_t^\alpha[a]\phi(t):= \int_{-\infty}^{a} \frac{\phi(t) - \phi(s)}{|t - s|^{1 + \alpha}} ds,
\qquad
\partial_t^\alpha[a,b]\phi(t)=  \int_{a}^b \frac{\phi(t) - \phi(s)}{|t - s|^{1 + \alpha}} ds.
\end{eqnarray*}
Notice that $\partial_t^\alpha \phi(t)$ is well-defined as soon as $\phi\in L_{\rm loc}^1(0,+\infty)$
and $\phi$ is $C^1$ in a neighborhood of $t$.
More about the functional formulation of this operator can be found in the references~\cite{kst06, diethelm10, gkmr14}. For the definition of viscosity solutions to~\eqref{eq}, we refer to~\cite{bi08, ty17}.

Notice that for all $\beta >0$, there exists $c_{\alpha,\beta}>0$ such that
\begin{eqnarray}\label{cst-deriv}
  \partial_t^\alpha t^{\beta} = c_{\alpha,\beta} t^{\beta -\alpha},
  \quad\text{for all $t\geq 0$. \quad \cite[Appendix B]{diethelm10}}
\end{eqnarray}
We introduce the Mittag-Leffler functions $E_\alpha (z)$ of order $\alpha$
as in~\cite[Definition 4.1]{diethelm10}. 
Recall that $E_\alpha$ is smooth on $\R$, $E_\alpha(0)=1$ and we have the following useful properties
\begin{eqnarray}
&& t\in [0,+\infty)\mapsto E_\alpha (-t) \text{ is positive, convex and nonincreasing, \cite[Section 6.2]{simon14},}\\
&&  \label{mlf}
\partial_t^\alpha E_\alpha (\lambda t^\alpha)=\lambda E_\alpha (\lambda t^\alpha),
\text{ for all $t> 0$ and $\lambda\in\R$, \cite[Theorem 4.3]{diethelm10},}\\
&& \label{estim-mlf}\frac{1}{\Gamma (1-\alpha)t}\leq E_\alpha (-t)\leq \frac{1}{\Gamma (1+\alpha)^{-1}t}
\text{ for all $t\geq 0$, \cite[Theorem 4]{simon14}.}
\end{eqnarray}
We will write $\widehat{x}=x/|x|$ for every $x\in\R^N\setminus\{0\}$.

%%%%%%%%%%%%%%%%%%%%%%%%%%%%%%%%%%%%%%%%%%%%%%%%%%%%%%%%%%%%%%%%%%%%%%%%%%%%%
%%%%%%%%%%%%%%%%%%%%%%%%%%%%%%%%%%%%%%%%%%%%%%%%%%%%%%%%%%%%%%%%%%%%%%%%%%%%%

\section{Regularity Estimates}
\label{sec:reg}
%The following result can be readily adapted from the arguments given in~\cite{ty17} in the case of initial condition at time $\tau = 0$.
%\begin{prop}[Comparison Principle]\label{propcomparacion}
%Let $\tau, T \in \R$ with $\tau < T$, and $u, v$ be bounded viscosity sub and supersolution to the problem
%\begin{equation}\label{eqtau}
%\begin{split}
%_\tau \! \partial^\alpha u - \I u + H(x, Du) = 0 \quad \mbox{in} \ \R^N \times (\tau, T]
%\end{split}
%\end{equation}
%with $u(x,\tau) \leq v(x, \tau)$ for all $x \in \R^N$. Then, $u \leq v$ in $\R^N \times [\tau, T]$.
%
%{\bb Moreover, if $u, v: (-\infty, T]$ are bounded viscosity sub and supersolution to problem~\eqref{eqtau} with $\tau = -\infty$, and $u, v$ satisfy
%\begin{align*}
%\limsup_{t \to -\infty} u(x,t) \leq \liminf_{t \to -\infty} v(x,t)
%\end{align*}
%uniformly in $x \in \R^N$, then $u \leq v$ in $\R^N \times (-\infty, T]$.
%}
%\end{prop}

We start with a regularity result in time for bounded solutions to~\eqref{eq}. This result is a consequence of the H\"older estimates reported by Barles, Chasseigne and Imbert in~\cite{bci11}.

%%%%
\begin{teo}\label{teoHoldert}
Assume~\eqref{H} or~\eqref{Hsuper}, and $g \in \mathrm{Lip}(\T^N)$.
Let $u$ be bounded, continuous  viscosity solution to~\eqref{eq}-\eqref{u0}.
Then, $u$ is H\"older continuous in time, that is there exists a constant $L > 1$ large enough such that
\begin{equation}\label{holder-en-t}
|u(x,s) - u(x,t)| \leq L |s - t|^\alpha, \quad s, t \geq 0.
\end{equation}
The constant $L$ depends on the data and $||u||_\infty$ but does not depend on $t$.
\end{teo}
%%%%

%%%%%%%%%%%%%%%%%%%%%%%%%%%%%%
\begin{proof}
Despite the proof we present here is valid for both~\eqref{H} and~\eqref{Hsuper}, we
underline the arguments in the later case. Let $\mu \in (0,1]$ be away from zero, and denote $\bar u = \mu u$. If~\eqref{Hsuper} is assumed, then we consider $\mu < 1$, and if~\eqref{H} is assumed, then $\mu = 1$ in what follows.

Using the linearity of the Caputo derivative, it is direct to check that $\bar u$ solves
\begin{equation}\label{eqmu}
\partial^\alpha_t \bar u + \mu H(x, \mu^{-1}D\bar u) = 0 \quad \mbox{in} \ Q,
\end{equation}
with initial condition $\bar u(\cdot, 0)=\mu g$.

By~\eqref{cst-deriv} and $g \in \mathrm{Lip}(\T^N)$,
we have that $g(x) \pm Ct^\alpha$
are respectively super- and subsolutions of~\eqref{eq}-\eqref{u0} for $C > 0$ large enough
depending only on $g$ and $H$.
Therefore, a direct application of the comparison results in~\cite{ty17} leads to the estimates
\begin{equation}\label{subsupsol}
g(x) - Ct^\alpha \leq u(x,t) \leq g(x) + C t^\alpha, \quad \mbox{for all} \ t \geq 0.
\end{equation}
These bounds can be readily adapted to $\bar u$ by multiplying by $\mu$ the last inequality.
 
By contradiction, assume that for all $L > 1$ we have
\begin{align}\label{contra1}
\sup_{x \in \T^N, s, t  \geq 0} \{ u(x,s) - u(x,t) - L|s - t|^\alpha \} =: \theta_L > 0.
\end{align}
Taking $\mu$ very close to $1$ in term of $\theta_L$ above, we can get
\begin{equation*}
\sup_{x \in \T^N, s, t  \geq 0} \{ \bar u(x,s) - u(x,t) - L|s - t|^\alpha \} = \theta_L/2 > 0.
\end{equation*}
%which is the perturbed version of the original contradiction condition~\eqref{contra1}. {\br Thus, by abuse of notation we write $u$ instead of $\bar u$ in the analysis below (and that can be easily followed by its evaluation in $(x,s)$).}

For localization purposes, we introduce a function $\psi_\beta$ with the following properties: we consider $\psi:\R \to \R$ smooth and nondecreasing, such that $\psi(t) = 0$ if $t \leq 1$, $\psi(t) \geq 2||u||_\infty$ if $t \geq 2$, and for $\beta > 0$ small we denote $\psi_\beta(t) = \psi(\beta t)$. Then, for all $\beta$ small enough, we have
\begin{align*}
\max_{x \in \R^N, s, t  \in \R} \{ \bar u(x,s) - u(x,t) - L|s - t|^\alpha - \psi_\beta(s) \} = \theta_L/4  > 0,
\end{align*}
where we recall that we extend $u(x,t)$ as $g$ for negative times $t$.

Then, for $\beta, \epsilon > 0$ small we define
\begin{align*}
\Phi(x,y,s,t) := \bar u(x,s) - u(y,t) - L|s - t|^\alpha - \epsilon^{-2}|x - y|^2 - \psi_\beta(s).
\end{align*}
We have that $\Phi$ attains its maximum at a point $(\bar x, \bar y, \bar s, \bar t) \in \bar Q^2$
and this maximum is bigger than $\theta_L/4$. 

Standard arguments lead to 
$$
|\bar x - \bar y| \leq \epsilon \omega_\beta( \epsilon), \quad |\bar s|, |\bar t| \leq 2/\beta,
$$ 
where $\omega_\beta(\epsilon) \to 0$ as $\epsilon \to 0$ if $\beta$ is fixed ($\omega_\beta$ is a modulus
of continuity in space of $u$ in the compact set $\T^N \times [0, 2/\beta]$), and 
\begin{align}\label{s-t1}
|\bar s - \bar t| \leq C_0 L^{-1/\alpha},
\end{align}
for some constant $C_0 > 0$ just depending on $||u||_\infty$. In particular, for $L$ large enough we may assume that $|\bar s - \bar t| < 1$.

Below, we use a
constant $C$, which may vary line to line but only depend on the data of the problem and not on $\epsilon, \beta$ nor $\mu$.

%since $u$ is continuous in $[0,T]$ we have that for $L, \beta$ fixed.

Here we claim that $\bar s, \bar t > 0$ for all $L$ large, and $\epsilon$ small in terms of $L$. In fact, if $\bar s = 0$ (the case $\bar t = 0$ is analogous), then, using~\eqref{subsupsol}, we see that
\begin{align*}
  \theta_L/4 < \Phi(\bar x, \bar y, \bar s, \bar t) \leq \mu g(\bar x) - g(\bar y) + (C_0-L) \bar t^\alpha
  \leq  C_0 L_0 \epsilon \ \omega_\beta( \epsilon) + C_0 (1 - \mu) + (C_0-L) \bar t^\alpha,
\end{align*}
where $L_0$ is the Lipschitz constant of $g$.
Taking $L\geq C_0$, we arrive at
\begin{align*}
\theta_L/4 \leq C \epsilon + C(1 - \mu),
\end{align*}
which is a contradiction if $\epsilon$ is taken small, and $\mu$ close to $1$ in terms of $L$.

In addition, since $u$ is uniformly continuous in space, then $\bar s \neq \bar t$  for all $L$ large enough and $\epsilon$ small in terms of $L$. Indeed, if $\bar s = \bar t$, then we would have
\begin{equation*}
0 < \theta_L /4 \leq \bar u(\bar x, \bar t) - u(\bar y, \bar t) \leq \omega_\beta(\epsilon) + C_0 (1 - \mu).
\end{equation*}
Hence, for $L$ and $\beta$ fixed,
taking $\epsilon$ small and $\mu$ close to $1$, we arrive at a contradiction.

%{\br Moreover, we claim that for all $L$ large enough and $\epsilon$ small in terms of $L$, we have 
%\begin{align}\label{claim1}
%|\bar s - \bar t|/2 \leq (\bar s \wedge \bar t).
%\end{align}
%
%To prove this, without loss of generality we can assume that $\bar s \wedge \bar t = \bar s < \bar t = \bar s \vee \bar t$. Notice that we also can assume that $\bar s, \bar t < 1$ since otherwise the claim follows at once in view of~\eqref{s-t1}.
%Using the maximality of $(\bar x, \bar y, \bar s, \bar t)$ and taking $\epsilon$ very small in terms of $L$ we have
%\begin{align*}
%0 \leq 2C_0 \bar t^\alpha - L|\bar s - \bar t|^\gamma, 
%\end{align*}
%that is
%\begin{align*}
%|\bar s - \bar t| \leq (2C_0/L)^{1/\gamma} \bar t^{\alpha/\gamma}.
%\end{align*}
%
%In particular, taking $L$ large enough in terms of $C_0$ and $\gamma$, and since $\gamma < \alpha$, we can say that $|\bar s - \bar t| \leq \bar t$.
%Moreover, using that $\bar t - \bar s \leq |\bar s - \bar t|$ we can also write
%\begin{align*}
%\bar t - \bar s \leq (2C_0/L)^{1/\gamma} \bar t,
%\end{align*}
%from which we get
%\begin{align*}
%\bar t (1 - (2C_0/L)^{1/\gamma}) \leq \bar s,
%\end{align*}
%and taking $L$ large in terms of $C_0$ and $\gamma$ we obtain $1/2 \bar t \leq \bar s$. This concludes~\eqref{claim1}.}

Since $\bar s, \bar t > 0$ and $\bar s \neq \bar t$, we can use the penalization defining $\Phi$ as a testing for $u$. For this, in what follows we write $\phi(x,y,s,t) := L|s - t|^\alpha + \epsilon^{-2}|x - y|^2 + \psi_\beta(s)$, from which we see that 
$$
\Phi(x,y,s,t) = \bar u(x,s) - u(y,t) - \phi(x,y,s,t).
$$
At this point, denoting $\phi_1(x,s) = \phi(x, \bar y, s, \bar t)$ we notice that the function
\begin{align*}
(x,s) \mapsto \bar u(x,s) - (u(\bar y, \bar t) + \phi_1(x,s))
\end{align*}
has a local maximum point at $(\bar x, \bar s)$, from which we can use the subsolution's viscosity inequality to write, for all $\delta > 0$, that
%Thus, we can use $(\bar x, \bar s)$ test point for $u$ with test function $\phi_1(x,s) := \phi(x, \bar y, s, \bar t)$  to conclude that for each $\delta_1, \rho > 0$ we can write
\begin{align*}
\partial_t^\alpha[\bar s - \delta]\bar u(\bar x, \cdot)(\bar s) + \partial_t^\alpha[\bar s - \delta, \bar s]\phi_1(\bar x, \cdot)(\bar s)  + \mu H(\bar x, \mu^{-1} D\phi_1(\bar x, \bar s)) \leq 0. 
\end{align*}
Similarly, denoting $\phi_2(y,t) = \phi(\bar x, y, \bar s, t)$ we notice the function
\begin{align*}
(y,t) \mapsto u(y,t) - (\bar u(\bar x, \bar s) - \phi_2(y,t))
\end{align*}
has a local minimum point at $(\bar y, \bar t)$, from which we can use the supersolution's viscosity inequality to write,
for all $\delta > 0$, that
\begin{align*}
  \partial_t^\alpha[\bar t - \delta]u(\bar y, \cdot)(\bar t) + \partial_t^\alpha[\bar t - \delta, \bar t](-\phi_2)(\bar y, \cdot)(\bar t)
  + H(\bar y, D(-\phi_2)(\bar y, \bar t)) \geq 0. 
\end{align*}
Then, we subtract both inequalities to arrive at
\begin{align}\label{testeo}
\mathcal D \leq \mathcal H,
\end{align}
where, noticing that $D(-\phi_2)(\bar y, \bar t) -=D\phi_1(\bar x, \bar s) = 2(\bar x-\bar y)/\epsilon^2$, we write
\begin{align*}
\mathcal D = & \partial_t^\alpha[\bar s - \delta]\bar u(\bar x, \cdot)(\bar s) + \partial_t^\alpha[\bar s - \delta, \bar s]\phi_1(\bar x, \cdot)(\bar s)
 - \partial_t^\alpha[\bar t - \delta]u(\bar y, \cdot)(\bar t) + \partial_t^\alpha[\bar t - \delta, \bar t]\phi_2(\bar y, \cdot)(\bar t).\\
%I = & \I_\rho^\rho(u, \phi_1, \bar x, \bar s) - \I_\rho^\rho(u, \phi_2, \bar y, \bar t) \\
 \mathcal H = & 
 H(\bar y,p)-\mu H(\bar x, \mu^{-1} p), \quad p:=2(\bar x-\bar y)/\epsilon^2.
\end{align*}
From~\eqref{Hsuper}, we get 
\begin{align*}
\mathcal H \leq -(1 - \mu) c_H^{-1} |p|^m + c_H (1+|p|^{m}) \epsilon  + A(1 - \mu)
%{\br \leq C\epsilon^m (1-\mu)^{-m+1}  + A(1 - \mu),}
\end{align*}
and from here, taking $\epsilon \leq (1 - \mu)c_H^{-2}$, we conclude that
\begin{align}\label{cotaH}
\mathcal H \leq C \epsilon + A(1 - \mu),
\end{align}
where $C > 0$ just depends on $c_H$.

\medskip

Consider $\delta = |\bar s - \bar t|/2 > 0$ and we split the term $\mathcal D$ as
\begin{align*}
\mathcal D = \mathcal D_1 + \mathcal D_2 
\end{align*}
with
\begin{align*}
\mathcal D_1 = & \partial_t^\alpha[\bar s - \delta]\bar u(\bar x, \cdot)(\bar s) - \partial_t^\alpha[\bar t - \delta]u(\bar y, \cdot)(\bar t) \\
\mathcal D_2 = & \partial_t^\alpha[\bar s - \delta, \bar s]\phi_1(\bar x, \cdot)(\bar s) + \partial_t^\alpha[\bar t - \delta, \bar t]\phi_2(\bar y, \cdot)(\bar t).
\end{align*}
We start with $\D_2$. Directly from the definitions, we see that
\begin{align*}
\D_2 = & L\int_{\bar s - \delta}^{\bar s} \frac{|\bar s - \bar t|^\alpha - |z - \bar t|^\alpha}{|\bar s - z|^{1 + \alpha}}dz + L\int_{\bar t - \delta}^{\bar t} \frac{|\bar s  - \bar t|^\alpha - |\bar s - z|^\alpha}{|\bar t - z|^{1 + \alpha}}dz + \int_{\bar s - \delta}^{\bar s} \frac{\psi_\beta(\bar s) - \psi_\beta(z)}{|\bar s - z|^{1 + \alpha}}dz \\
=: & \D_2' + o_\beta(1),
\end{align*}
where $o_\beta(1) \to 0$ as $\beta \to 0$ independent of the rest of the variables by the smoothness of $\psi$.
Now, performing the change of variables $z = \bar s + y$ in the first integral, and $z = \bar t + y$ in the second, we arrive at
\begin{align*}
\D_2' &=  L\int_{- \delta}^{0} \frac{|\bar s - \bar t|^\alpha - |\bar s - \bar t + y|^\alpha}{|y|^{1 + \alpha}}dy + L\int_{- \delta}^{0} \frac{|\bar s  - \bar t|^\alpha - |\bar s - \bar t - y|^\alpha}{|y|^{1 + \alpha}}dy \\
%&= & L\int_{- \delta}^{0} \frac{|\bar s - \bar t|^\alpha - |\bar s - \bar t + y|^\alpha}{|y|^{1 + \alpha}}dy - L\int_{\delta}^{0} \frac{|\bar s  - \bar t|^\alpha - |\bar s - \bar t + y|^\alpha}{|y|^{1 + \alpha}}dy \\
&=  L\int_{- \delta}^{0} \frac{|\bar s - \bar t|^\alpha - |\bar s - \bar t + y|^\alpha}{|y|^{1 + \alpha}}dy + L\int_{0}^{\delta} \frac{|\bar s  - \bar t|^\alpha - |\bar s - \bar t + y|^\alpha}{|y|^{1 + \alpha}}dy \\
&=  -L\int_{- \delta}^{\delta} \frac{|\bar s - \bar t + y|^\alpha - |\bar s - \bar t|^\alpha}{|y|^{1 + \alpha}}dy\\
&= -L \int_{- \delta}^{\delta} \frac{|\bar s - \bar t + y|^\alpha - |\bar s - \bar t|^\alpha
  - \alpha |\bar s - \bar t|^{\alpha - 1}\widehat{\bar s - \bar t}\,y}{|y|^{1 + \alpha}}dy,
\end{align*}
where the last equality comes from the symmetry of the kernel.
Performing a second order expansion and recalling that $\delta = |\bar s - \bar t|/2$, we obtain that there
exists $\rho(y)\in (-\delta, \delta)$ such that
\begin{align*}
  - \left(|\bar s - \bar t + y|^\alpha - |\bar s - \bar t|^\alpha
    - \alpha |\bar s - \bar t|^{\alpha - 1}\widehat{\bar s - \bar t}\,y\right) 
 & = -\frac{\alpha (\alpha -1)}{2} |\bar s - \bar t + \rho(y)|^{\alpha - 2}y^2\\
 & \geq- \frac{\alpha (\alpha -1)}{2^{\alpha-1}} |\bar s - \bar t|^{\alpha - 2}y^2.
\end{align*}
Therefore,
$$
\D_2' \geq   -\frac{L \alpha (\alpha - 1)}{2^{\alpha-1}}  |\bar s - \bar t|^{\alpha - 2} \int_{- \delta}^{\delta}  |y|^{1-\alpha} dy
= \frac{\alpha (1- \alpha )}{2-\alpha}L=:  cL,
$$
and from here
\begin{align}\label{D2}
\D_2 \geq cL - o_\beta(1).
\end{align}
Now we deal with $\D_1$, writing
\begin{align*}
\D_1 = & \partial_t^\alpha [\bar s-1]\bar u(\bar x, \cdot)(\bar s) - \partial_t^\alpha[\bar t -1]u(\bar y, \cdot)(\bar t) \\
& + \partial_t^\alpha [\bar s -1, \bar s - \delta]\bar u(\bar x, \cdot)(\bar s) - \partial_t^\alpha[\bar t -1, \bar t - \delta]u(\bar y, \cdot)(\bar t) \\
\geq & -C ||u||_\infty +  \partial_t^\alpha [\bar s-1, \bar s - \delta]\bar u(\bar x, \cdot)(\bar s) - \partial_t^\alpha[\bar t -1, \bar t - \delta]u(\bar y, \cdot)(\bar t) \\
%\geq & -C |u(\bar x, 0) - u(\bar x, \bar s)| \bar s^{-\alpha}  -C |u(\bar y, 0) - u(\bar y, \bar t)| \bar t^{-\alpha} \\
%& \ + \partial^\alpha [0, \bar s - \delta]u(\bar x, \cdot)(\bar s) - \partial^\alpha[0, \bar t - \delta]u(\bar y, \cdot)(\bar t) \\
%\geq & -2 C C_0 + \partial^\alpha [0, \bar s - \delta]u(\bar x, \cdot)(\bar s) - \partial^\alpha[0, \bar t - \delta]u(\bar y, \cdot)(\bar t) \\
=: & - C ||u||_\infty + \D_1',
\end{align*}
where $C > 0$ depends only on $\alpha$.

To deal with the remaining term $\D_1'$, we assume that $\bar s < \bar t$ (the case $\bar t < \bar s$ follows the same lines). Performing similar change of variables as above and using the maximal inequality $\Phi(\bar x, \bar y, \bar s, \bar t) \geq \Phi(\bar x, \bar y, \bar s + y, \bar t + y)$ we arrive at
\begin{align*}
\D_1' = & \int_{-1}^{-\delta} \frac{\bar u(\bar x, \bar s) - u(\bar y, \bar t) - (\bar u(\bar x, \bar s + y) - u(\bar y, \bar t + y))}{|y|^{1 + \alpha}}dy
\geq \int_{-1}^{-\delta} \frac{\psi_\beta(\bar s + y) -\psi_\beta(\bar s)}{|y|^{1 + \alpha}}dy.
%& \ - \int_{-\bar t}^{-\bar s} \frac{u(\bar y, \bar t) - u(\bar y, \bar t + y)}{|y|^{1 + \alpha}}dy.
\end{align*}
Noticing that the smooth function $\psi_\beta$
satisfies $|\psi_\beta '|\leq C\beta$, we conclude that $\D_1' \geq -o_\beta(1)$.
From this
%meanwhile, using the maximal inequality $\Phi(\bar x, \bar y, \bar s, \bar t) \geq \Phi(\bar x, \bar y, \bar s, \bar t + y)$ for $y \geq -\bar t$ we arrive at
%\begin{align*}
%D_0 \geq -L \int_{-\bar t}^{-\bar s} \frac{|\bar s - \bar t|^\gamma - |\bar s- \bar t - y|^\gamma}{|y|^{1 + \alpha}}dy \geq 0.
%\end{align*}
%
%The above estimates lead us to
\begin{align*}
\D_1 \geq -C ||u||_\infty - o_\beta(1).
\end{align*}
Joining this with~\eqref{cotaH} and~\eqref{D2} in~\eqref{testeo} we get
\begin{align*}
cL \leq C||u||_\infty + o_\beta(1) + C \epsilon + A(1 - \mu).
\end{align*}
Then, we let $\epsilon \to 0$ first, then $\mu \to 1$ and finally $\beta \to 0$ and,
having taken $L$ large enough just in terms of the data and $||u||_\infty$,
we reach a contradiction. It ends the proof.
\end{proof}
%%%%%%%%%%%%%%%%%%%%%%%%%%%%%%

\medskip
Now we would like to obtain estimates in space.
\begin{teo}\label{holderx}
Assume hypotheses of Theorem~\ref{teoHoldert} hold.
For each bounded viscosity solution to~\eqref{eq}, there exists a modulus
of continuity $m\in C([0,+\infty))$ independent of $t$ such that
\begin{equation*}
|u(x,t) - u(y,t)| \leq m(|x - y|) \quad \mbox{for all} \ x, y \in \T^N, \ t \geq 0.
\end{equation*}

If, in addition, $H$ satisfies~\eqref{coerc-Hsuper} for some $m \geq 1$, then,
for each $\beta \in (0,1)$, there exists a constant $C > 0$ such that 
\begin{equation*}
  |u(x,t) - u(y,t)| \leq C|x - y|^\beta \quad \mbox{for all} \ x, y \in \T^N, \ t \geq 0.
\end{equation*}

The modulus $m$ and the constant $C$ depend on the datas and $||u||_\infty$ but do not depend on $t$.
\end{teo}
%%%%%%%%%%%

\begin{proof}
For $\epsilon\in (0,1)$, we introduce
the sup-convolution
$$
u^\epsilon(x,t) = \sup_{s \geq 0} \{ u(x,s) - \epsilon^{-1}|s - t|^2 \}.
$$
We collect some properties of this regularization of $u$:
\begin{itemize}
\item[(i)] $u^\epsilon$ is still H\"older continuous in time satisfying~\eqref{holder-en-t} like $u$,
\item[(ii)] $||u^\epsilon-u||_\infty\leq  C\epsilon^{\alpha/2}$,
\item[(iii)] $u^\epsilon$ is Lipschitz continuous in time with Lipschitz constant  $C\epsilon^{-1}$ with $C>0$ just depending on $||u||_\infty$,
\item[(iv)] $u^\epsilon$ is a viscosity subsolution to~\eqref{eq} in $\T^N \times (a_\epsilon, +\infty)$ for some $a_\epsilon \to 0$ as $\epsilon \to 0$.
\end{itemize}
%For the proof of (i), if $\bar t$ denotes the point where the maximum is achieved in $u^\epsilon(x,t)$, then
%we have, for all $s\geq 0$,
%$$
%u^\epsilon(x,t)-u^\epsilon(x,t')\leq u(x,\bar t)-\frac{|t-\bar t|^2}{\epsilon}-u(x,s)+\frac{|t'-s|^2}{\epsilon}.
%$$
%Writing the above inequality for $s=\bar t +t'-t$ and using~\eqref{holder-en-t}, we obtain (i).
%To prove (ii), we use again~\eqref{holder-en-t},
%$$
%u^\epsilon(x,t)-u(x,t)=\sup_{s \geq 0} \{ u(x,s) - u(x,t)-\frac{|s-t|^2}{\epsilon} \}
%\leq \sup_{s \geq 0} \{C|t-s|^\alpha -\frac{|s-t|^2}{\epsilon}\}\leq C\epsilon^{\frac{\alpha}{2-\alpha}}\leq  C\epsilon^{\alpha/2}.
%$$
%The proof of (iii) is a classical property of the sup-convolution regularization and (iv) is proved in~\cite{ty17}.}
The proofs of (i) and (ii) are easy consequences of Theorem~\ref{teoHoldert},~(iii)
is a classical property of the sup-convolution regularization and (iv) is proved in~\cite{ty17}.

%%%%%%%%%%%
Now we prove the desired regularity of $u$ by adapting the standard viscosity procedure to get regularity estimates from the coercivity of $H$.
For any $x_0\in\T^N$, $s_0>0$, $\beta>0$ and $L>0$ to be chosen, we consider
\begin{equation}\label{max620}
\sup_{x \in \T^N, s \geq 0} \{ u^\epsilon(x, s) - u^\epsilon(x_0, s_0) - L|x - x_0|-\frac{|s-s_0|^2}{\beta^2} \},
\end{equation}
where $\epsilon$ is chosen small enough in order that $s_0>a_\epsilon$ in (iv). Classical results imply
that this maximum is achieved at $(\bar x,\bar s)$ with $\bar s\to s_0$ as $\beta \to 0$.
We take $\beta$ small enough in order that $\bar s> a_\epsilon$.

If $\bar x\not= x_0$, then we use $(x,s)\mapsto u^\epsilon(x_0, s_0) + L|x - x_0|+\frac{|s-s_0|^2}{\beta^2}$
as a test function for the subsolution $u^\epsilon$ of~\eqref{eq} at $(\bar x,\bar s)$ to get that,
for every $\bar \delta\in (0,1)$,
\begin{equation}\label{ineg-visco845}
\partial_t^\alpha[\bar s-\bar \delta] u^\epsilon(\bar x, \cdot)(\bar s)
+ \partial_t^\alpha[\bar s-\bar \delta, \bar s] \frac{|\cdot-s_0|^2}{\beta^2}(\bar s)
+H(\bar x, L \widehat{\bar x - x_0}) \leq 0.
\end{equation}

Actually, since  $\frac{|\cdot-s_0|^2}{\beta^2}$ is smooth and $u^\epsilon$
is Lipschitz continuous, we can send $\bar\delta\to 0$ in the previous inequality.
In other words, due to the Lipschitz continuity of $u^\epsilon$, 
we can use $u^\epsilon$ itself as a test-function in the fractional derivative in the viscosity inequality,
see~\cite[Proposition 2.4]{ty17} for details.       

It follows that it is enough to estimate
the fractional term $\partial_t^\alpha u^\epsilon(\bar x, \cdot)(\bar s)$
that we expand, for $\delta>0$, as
$$
\partial_t^\alpha u^\epsilon(\bar x, \cdot)(\bar s)
= \partial_t^\alpha[\bar s-1] u^\epsilon(\bar x, \cdot)(\bar s)
+\partial_t^\alpha[\bar s-1, \bar s-\delta] u^\epsilon(\bar x, \cdot)(\bar s)
+ \partial_t^\alpha[\bar s-\delta, \bar s] u^\epsilon(\bar x, \cdot)(\bar s).
$$

At first, from (ii),
$$
\partial_t^\alpha[\bar s-1] u^\epsilon(\bar x, \cdot)(\bar s)\geq -C||u^\epsilon||_\infty\geq -C(||u||_\infty+\epsilon^{\alpha/2}).
$$
Then, using (i), 
\begin{eqnarray*}
\partial_t^\alpha[\bar s-1, \bar s-\delta] u^\epsilon(\bar x, \cdot)(\bar s)
=\int_{\bar s-1}^{\bar s-\delta}  \frac{u^\epsilon(\bar x, \bar s)- u^\epsilon(\bar x, s)}{|\bar s-s|^{1+\alpha}}ds
\geq -C\int_{\bar s-1}^{\bar s-\delta} \frac{1}{|\bar s-s|}ds\geq -C|\log(\delta)|.
\end{eqnarray*}
For the third term, we use (iii) to obtain
\begin{eqnarray*}
\partial_t^\alpha[\bar s-\delta, \bar s] u^\epsilon(\bar x, \cdot)(\bar s)
\geq -C\int_{\bar s-\delta}^{\bar s} \frac{1}{|\bar s-s|^\alpha}ds \geq -\frac{C}{\epsilon}\delta^{1-\alpha}.
\end{eqnarray*}
%Finally
%\begin{eqnarray*}
%\partial_t^\alpha[\bar s-\bar \delta, \bar s] \frac{|\cdot-s_0|^2}{\beta^2}(\bar s) 
%&=&  \frac{1}{\beta^2}\int_{\bar s-\bar\delta}^{\bar s}\frac{|\bar s-s_0|^2-|s-s_0|^2}{|\bar s -s|^{1+\alpha}}ds\\
%&\geq& -\frac{1}{\beta^2}\int_{\bar s-\bar\delta}^{\bar s}\frac{|\bar s -s|^2 +2|\bar s -s| |\bar s -s_0|}{|\bar s -s|^{1+\alpha}}ds\\
%&\geq& -\frac{1}{\beta^2}\int_{\bar s-\bar\delta}^{\bar s}|\bar s -s|^{1-\alpha}ds
%-\frac{|\bar s -s_0|}{\beta^2}\int_{\bar s-\bar\delta}^{\bar s}\frac{1}{|\bar s -s|^{\alpha}}ds\\
%&\geq& -C(\beta^{-2}\bar\delta^{2-\alpha}+\epsilon^{-1}\bar\delta^{1-\alpha}),
%\end{eqnarray*}
%where we used that ${|\bar s -s_0|}/{\beta^2}$ is bounded by $C\epsilon^{-1}$ as a consequence of (iii).
%Notice that it is possible to send $\bar{\delta}$ to 0 in the above inequality, which means actually that
%we can use $u^\epsilon$ itself as a test-function in the fractional term in the viscosity inequality~\eqref{ineg-visco845} because $u^\epsilon$
%is Lipschitz continuous (see~\cite[Proposition 2.4]{ty17}).

Plugging these estimates in~\eqref{ineg-visco845}, we obtain
\begin{equation*}
  H(\bar x, L \widehat{\bar x - x_0}) \leq  C(1 + |\log(\delta)| + \epsilon^{-1} \delta^{1 - \alpha},
\end{equation*}
where $C > 0$ just depends on the data and $||u||_\infty$. 
Taking the minimum for $\delta >0$ we arrive at
\begin{equation}\label{inegH780}
H(\bar x, L \widehat{\bar x - x_0}) \leq  C(1 + |\log(\epsilon)|).
\end{equation}
From the coercivity of $H$, we reach a contradiction if $L=L(\epsilon)$ is large enough.

It follows that the maximum in~\eqref{max620} is achieved for $\bar x=x_0$, which implies
$$
u^\epsilon (x,\bar s)-u^\epsilon (x_0,\bar s)\leq L(\epsilon)|x-x_0|, \qquad\text{for all $x\in\T^N$.}
$$
Sending  $\beta\to 0$, recalling $\bar s\to s_0$ as $\beta\to 0$
and that $x_0, s_0$ are arbitrary, we finally obtain that,
for all $x,y\in\T^N$, $t>0$,
\begin{eqnarray}\label{abc3765}
&& u(x,t) - u(y,t) \leq u^\epsilon(x,t) - u^\epsilon(y, t)+C\epsilon^{\alpha/2}\leq L(\epsilon) |x-y| + C\epsilon^{\alpha/2},
\quad 0<\epsilon <1.
\end{eqnarray}
This latter inequality means that $u$ is uniformly continuous with respect to $x$ independently of $t$.

In addition, if $H$ satisfies~\eqref{coerc-Hsuper} for some $m \geq 1$, then~\eqref{inegH780}
and~\eqref{abc3765} lead to
\begin{eqnarray*}
&& u(x,t) - u(y,t) \leq C(1+|\log(\epsilon)|) |x-y| + C\epsilon^{\alpha/2},
\quad 0<\epsilon <1.
\end{eqnarray*}
Thus, taking the infimum with respect to $0< \epsilon <1$, we conclude that
\begin{equation*}
|u(x,t) - u(y, t)| \leq C (1+ |\log |x - y| |)  |x - y|,
\end{equation*}
from which the result follows.
\end{proof}

\begin{remark}
When $H$ has a sublinear growth,
it is possible to obtain some better regularity estimates in space, namely,
Lipschitz estimates. More precisely, if~\eqref{H} holds and $g \in \mathrm{Lip}(\T^N)$,
then every bounded viscosity solution to~\eqref{eq} satisfies
\begin{equation*} %\label{sollip111}
|u(x,t) - u(y,t)| \leq (1+{\rm Lip}(g))\, E_\alpha\left( 2 c_H t^\alpha\right) |x-y| \quad \mbox{for all} \ x,y \in \T^N, t \geq 0.
\end{equation*}
Such a result was already showed in Giga and Namba~\cite{gn17}. We do not focus on such results because
the Lipschitz constant depends heavily on time, a dependence we want to avoid in order to obtain
the large time behavior of the solution.
\end{remark}

%%%%%%%%%%%%%%%%%%%%%%%%%%%%%%%%%%%%%%%%%%%%%%%%%%%%%%%%%%%%%%%%%%
%%%%%%%%%%%%%%%%%%%%%%%%%%%%%%%%%%%%%%%%%%%%%%%%%%%%%%%%%%%%%%%%%%
%%%%%%%%%%%%%%%%%%%%%%%%%%%%%%%%%%%%%%%%%%%%%%%%%%%%%%%%%%%%%%%%%%
%%%%%%%%%%%%%%%%%%%%%%%%%%%%%%%%%%%%%%%%%%%%%%%%%%%%%%%%%%%%%%%%%%

\section{Oscillating function with positive Caputo derivative.}
\label{seccounterexample}

In this section we construct a bounded function $u: [0,+\infty) \to \R$ such that $\partial_t^{\alpha}u \geq 0$ but such that
$$
\liminf_{t \to +\infty} u(t) < \limsup_{t \to +\infty} u(t),
$$
which prevents $u$ to have any limit as $t\to +\infty$. This result shows a great contrast with the standard case $\alpha = 1$ in which $\partial_t u\geq0$ implies that $u$ is a nondecreasing function and therefore it is convergent.

%It is a well-known fact that if $u: [0,+\infty) \to \R$ is a continuously differentiable function with positive first derivative then it is decreasing. Moreover, if it is also bounded, then $\lim_{t\to+\infty} u(t)$ exists and it is a finite real number.
%Of course, the last assertion requires a suitable integrability condition over $u'$. A function $u$ such that $u' \sim t^{-1}$ for all $t$ large is unbounded, meanwhile if $\epsilon > 0$ and $u' \sim -t^{1 + \epsilon}$ for all $t$ large, then $u$ is bounded. Clearly, Fundamental Theorem of Calculus leads directly to these simple facts. In the context of Caputo derivative, in~\cite{GM} it is possible to see that such a theorem has the more involved formulation
In what follows, for any $\alpha\in(0,1)$, we define the  incomplete regularized beta function (see~\cite[Chapter 6]{as64}) by
\begin{eqnarray*}
B_{\alpha}[z_0,z_1]:=\frac{1}{\pi\csc(\alpha\pi)}\int_{z_0}^{z_1}t^{-\alpha}(1-t)^{\alpha-1}dt, \quad
\text{for all $0\leq z_0\leq z_1\leq 1$,}
\end{eqnarray*}
and we simply write $B_\alpha[z] = B_\alpha[0,z]$. We remark that $B_{\alpha}[0,1]=1$ (\cite[6.1.17 and 6.2.2]{as64}).

We also define
\begin{eqnarray*}
	b_{\alpha}:=B^{-1}_{\alpha}[1/2]\in(0,1),
\end{eqnarray*}
where $B_\alpha^{-1}[\cdot]$ is the inverse function of $B_\alpha[\cdot]$. As an example, if $\alpha=1/2$, then $b_{\alpha}=1/2$. %Notice that $b_\alpha \geq 1/2$ if $\alpha \leq 1/2$.

Hence, in the general case, with this choice of $b_{\alpha}$, 
we have 
\begin{eqnarray}\label{b}
B_{\alpha}[0,b_{\alpha}]=B_{\alpha}[b_{\alpha},1]=1/2.
\end{eqnarray}
Then, we define
$$
\eta_{\alpha}:=\pi\csc(\alpha\pi)B_{\alpha}[b_{\alpha}^{3},b_{\alpha}^{2}] \in (0,1),
$$
and consider the continuous functions
\begin{equation*}
f_1(t) = \left \{ \begin{array}{cl} 1 \quad & \mbox{if} \ t \in [0,1], \\ t^{-\alpha} \quad & \mbox{if} \ t \geq 1, \end{array} \right .
\end{equation*}
and
\begin{equation*}
f_2(t) = \left \{ \begin{array}{cl}\frac{t-a_{2k}}{\epsilon_k} \quad & \mbox{if} \ t \in [a_{2k}, a_{2k}+\epsilon_k), \\ 1 \quad & \mbox{if} \ t \in [a_{2k}+\epsilon_k, a_{2k + 1}-\epsilon_k), \\ \frac{a_{2k+1}-t}{\epsilon_k} \quad & \mbox{if} \ t\in [a_{2k+1}-\epsilon_k, a_{2k + 1}),\\ 0 \quad & \mbox{if not,}\end{array} \right .
\end{equation*}
%$$
%_2(t) = \sum_{k=0}^{\infty} \mathbf{1}_{[a_{2k}, a_{2k + 1})}(t),
%$$
where 
\begin{equation*}
% a_{2k}:=(1/b_{\alpha})^{2k}
a_{k}:=(1/b_{\alpha})^{k}
\quad \mbox{and} \quad \epsilon_k:=\frac{1-b_{\alpha}^{2}}{4}\frac{\eta_{\alpha}}{a_{2k}}, \quad \mbox{for all} \ k \geq 0.
\end{equation*} 
Next, we consider the continuous function $f := f_1 f_2$ (see Figure~\ref{dess-f})
and we define $u$ as
\begin{equation}\label{defu}
u(t) = \int_{0}^{t} \frac{f(z)}{(t-z)^{1-\alpha}}dz, \quad \text{for all $t \geq 0$.}
\end{equation}
%%%%%%
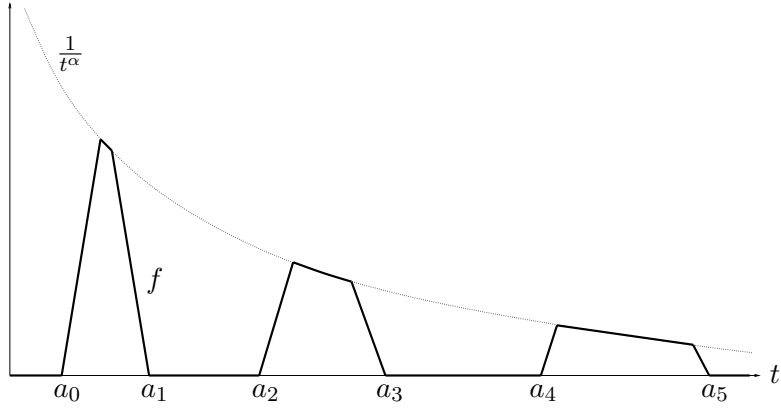
\begin{figure}[ht]
\begin{center}
\includegraphics[width=10cm]{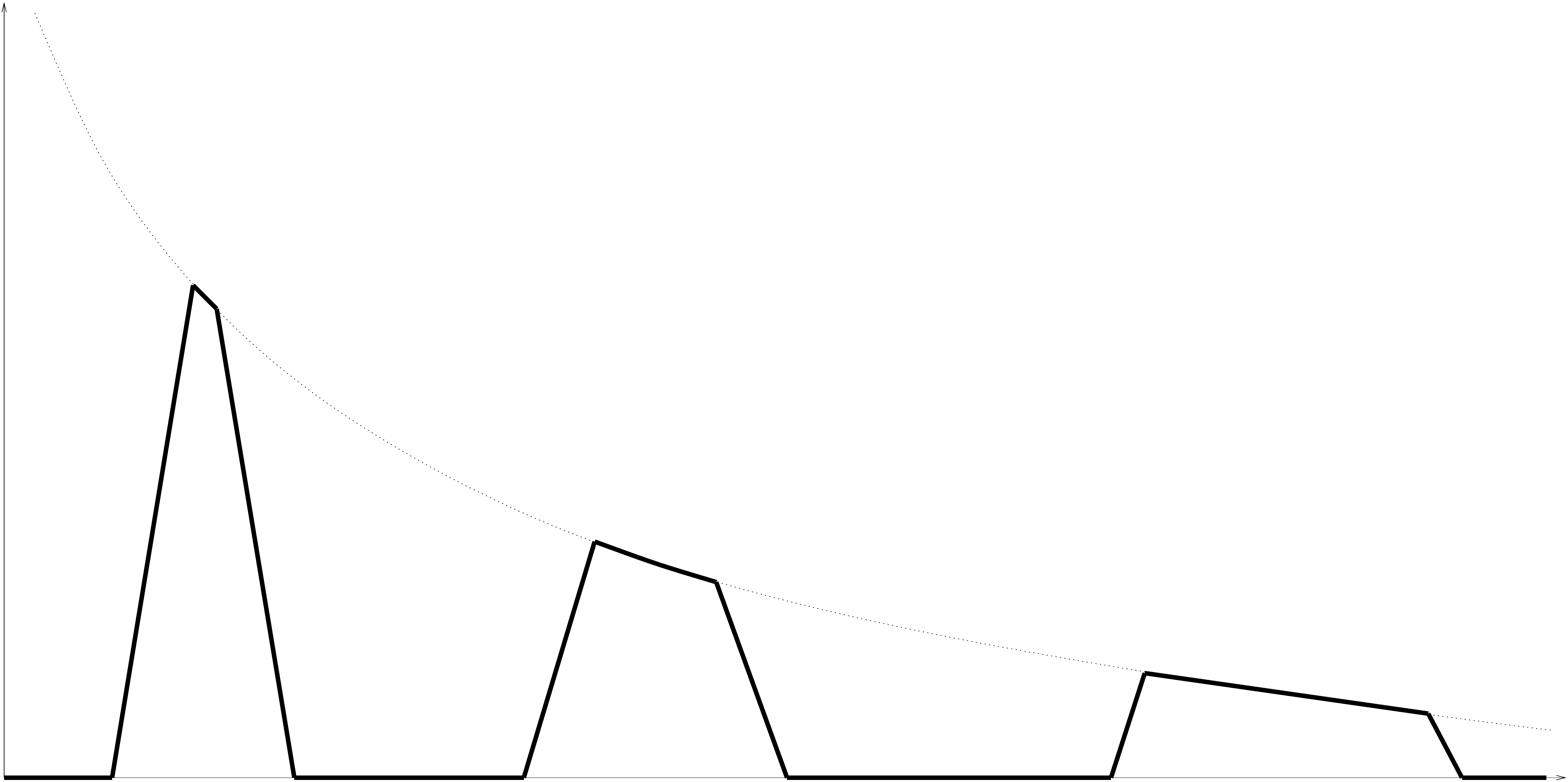}
\end{center}
\caption{Behavior of $f=f_1f_2$}
\label{dess-f}
\unitlength=1pt
\begin{picture}(0,0)
\put(146,38){$t$}
\put(-124,160){$\frac{1}{t^\alpha}$}
\put(-90,74){$f$}
\put(-125,33){$a_0$}
\put(-92,33){$a_1$}
\put(-50,33){$a_2$}
\put(-3,33){$a_3$}
\put(55,33){$a_4$}
\put(120,33){$a_5$}
\end{picture}
\end{figure}
%%%%%%%
%%%% DESSIN

The function $f$ is regular enough (locally Lipschitz) to use the representation formula in~\cite[Theorem 3.7]{diethelm10}, meaning that $u$ solves the fractional ODE
\begin{equation}\label{ode1}
\partial_t^{\alpha} u = f \quad \mbox{in} \ (0, +\infty), \qquad u(0) = 0,
\end{equation}
where $\partial_t^{\alpha}$ is the Caputo derivative of order $\alpha\in(0,1)$.

Notice that  $u = 0$ in $[0,1]$ and $0 \leq u$ is bounded in $\R^+$. In fact,
for $t > 1$ we see that
\begin{equation*}
u(t) \leq \int_{1}^{t} \frac{f_1(z)}{(t-z)^{1-\alpha}}dz,
\end{equation*}
from which we get that
\begin{equation*}
u(t) \leq t^{\alpha-1} \int_{1}^{t} \frac{z^{-\alpha}}{(1 - z/t)^{1-\alpha}} dz = t^{\alpha-1} \int_{1/t}^{1} \frac{t^{-\alpha}y^{-\alpha}}{(1 - y)^{1-\alpha}} tdy \leq \int_{0}^{1} \frac{dy}{y^{\alpha}(1 - y)^{1-\alpha}} =\pi\csc(\alpha\pi),
\end{equation*}
from which $u$ is bounded. 

Now we compare $u(a_{2N + 1})$ and $u(a_{2N + 2})$, for $N$ large enough such that
\begin{equation}\label{condN}
\int^{1}_{1-\epsilon_N/a_{2N+1}}\frac{dy}{y^{\alpha}(1 - y)^{1-\alpha}}<\eta_\alpha/4
\quad \text{and} \quad
a_{2N}(a_1-1)=(1/b_\alpha)^{2N}( 1/b_\alpha-1)\geq 2.
\end{equation}
Using the definition of $f$ we see that
\begin{align*}
u(a_{2N + 1})%=& \int_{0}^{a_{2N + 1}} \frac{f(z)}{(a_{2N + 1} - z)^{1-\alpha}}dz \\
= &\int_{1}^{a_{2N + 1}} \frac{z^{-\alpha}f_2(z)}{(a_{2N + 1} - z)^{1-\alpha}}dz \\
= & \sum_{k=0}^{N}\Bigg( \int_{a_{2k}}^{a_{2k}+\epsilon_k} \frac{z^{-\alpha}f_2(z)}{(a_{2N + 1} - z)^{1-\alpha}}dz+\int_{a_{2k}+\epsilon_k}^{a_{2k+1}-\epsilon_k} \frac{z^{-\alpha}}{(a_{2N + 1} - z)^{1-\alpha}}dz\\
&\hspace{7cm}+\int_{a_{2k+1}-\epsilon_k}^{a_{2k+1}} \frac{z^{-\alpha}f_2(z)}{(a_{2N + 1} - z)^{1-\alpha}}dz\Bigg) \\
= &\sum_{k=0}^{N}\int_{a_{2k}}^{a_{2k+1}} \frac{z^{-\alpha}}{(a_{2N + 1} - z)^{1-\alpha}}dz
- \sum_{k=0}^{N}\int_{a_{2k}}^{a_{2k}+\epsilon_k} \frac{z^{-\alpha}(1-f_2(z))}{(a_{2N + 1} - z)^{1-\alpha}}dz
\\
&\hspace{7cm}-\sum_{k=0}^{N}\int_{a_{2k+1}-\epsilon_k}^{a_{2k+1}} \frac{z^{-\alpha}(1-f_2(z))}{(a_{2N + 1} - z)^{1-\alpha}}dz \\
=: &v_1(a_{2N+1})-v_2(a_{2N+1})-v_3(a_{2N+1}),
\end{align*}
and similarly
$$u(a_{2N+2})=:v_1(a_{2N+2})-v_2(a_{2N+2})-v_3(a_{2N+2}).$$
From here, by simple integration we get
\begin{align*}
v_1(a_{2N + 1})= & \sum_{k=0}^{N} \int_{a_{2k}/a_{2N + 1}}^{a_{2k + 1}/a_{2N + 1}} \frac{dy}{y^{\alpha}(1 - y)^{1-\alpha}}\\
=&\pi\csc(\alpha\pi)\sum_{k=0}^{N} B_{\alpha}[a_{2k}/a_{2N + 1},a_{2k + 1}/a_{2N + 1}].
%=& -\sum_{k=0}^{N} [\arccos(2a_{2k + 1}/a_{2N + 1} - 1) - \arccos(2a_{2k}/a_{2N + 1} - 1)].
\end{align*}
Moreover,
\begin{equation*}
v_1(a_{2N + 2})=\pi\csc(\alpha\pi)\sum_{k=0}^{N} B_{\alpha}[a_{2k}/a_{2N + 2},a_{2k + 1}/a_{2N + 2}].
%-\sum_{k=0}^{N} [\arccos(2a_{2k + 1}/a_{2N + 2} - 1) - \arccos(2a_{2k}/a_{2N + 2} - 1)].
\end{equation*}
Now, we estimate the term $v_1(a_{2N + 2}) - v_1(a_{2N + 1})$. For this, using the definition of $b_{\alpha}$, $\eta_\alpha$
and~\eqref{b}, we notice that
%\begin{equation*}
%\beta_k(N) := \frac{a_{2k}}{a_{2N + 1}} = 2^{2(k - N) - 1} = \frac{a_{2k + 1}}{a_{2N + 2}},
%\end{equation*}
%and that
%\begin{align*}
%\gamma_{k - 1}(N) := \frac{a_{2(k - 1) + 1}}{a_{2N + 1}} = 2^{2(k - N) - 2} = \frac{a_{2k}}{a_{2N + 2}} =: \delta_k(N),
%\end{align*}
%from which we arrive at the expression
%\begin{align*}
%v_1(a_{2N + 2}) - v_1(a_{2N + 1})= \sum_{k=0}^{N - 1}[\arccos(2\gamma_k - 1) - 2\arccos(2\beta_k - 1) + \arccos(2\delta_k - 1)] - \pi/3.
%\end{align*}
\begin{align*}
\frac{v_1(a_{2N + 2}) - v_1(a_{2N + 1})}{\pi\csc(\alpha\pi)}=&\sum_{k=0}^{N} \left(B_{\alpha}[a_{2k}/a_{2N + 2},a_{2k + 1}/a_{2N + 2}]-B_{\alpha}[a_{2k}/a_{2N + 1},a_{2k + 1}/a_{2N + 1}]\right)\\
\leq& B_{\alpha}[0,b_{\alpha}]-\sum_{k=0}^{N} B_{\alpha}[a_{2k}/a_{2N + 1},a_{2k + 1}/a_{2N + 1}]\\
=& B_{\alpha}[0,b_{\alpha}]-B_{\alpha}[b_{\alpha},1]-\sum_{k=0}^{N-1} B_{\alpha}[a_{2k}/a_{2N + 1},a_{2k + 1}/a_{2N + 1}]\\
=&-\sum_{k=0}^{N-1} B_{\alpha}[a_{2k}/a_{2N + 1},a_{2k + 1}/a_{2N + 1}]\\
\leq&-B_{\alpha}[b_{\alpha}^{3},b_{\alpha}^{2}].
\end{align*}
Therefore, we have that $v_1(a_{2N + 2}) - v_1(a_{2N + 1})\leq-\eta_\alpha$.
%Now, notice that for all $k$ we have $\delta_k < \beta_k < \gamma_k$ and we recall that $\gamma_{k - 1} = \delta_k$. In addition, we see that $\gamma_{N-1} = 1/4$ and therefore $2 \gamma_{N - 1} - 1 = -1/2$ and thus $\arccos(2 \gamma_{N - 1} - 1) = 2\pi/3$. From this, it is possible to see that
%\begin{align*}
%v_1(a_{2N + 2}) - v_1(a_{2N + 1}) \leq & (\pi - 2\pi/3) + \arccos(2\gamma_{N - 1} - 1) - \arccos(2\beta_{N - 1} - 1) - \pi/3\\
%= & \arccos(2^{-2} - 1) - \arccos(2^{-3} - 1)\\ 
%=&-\eta_\alpha.
%\end{align*}
Hence, using the above result and the fact that $v_2, v_3\geq0$, we have that
\begin{align*}
u(a_{2N + 2}) - u(a_{2N + 1})\leq&-\eta_\alpha+v_2(a_{2N+1})+v_3(a_{2N+1}).
\end{align*}
Finally, we estimate the last two terms
\begin{align*}
v_2(a_{2N+1})=&\sum_{k=0}^{N}\int_{a_{2k}}^{a_{2k}+\epsilon_k} \frac{z^{-\alpha}(1-\frac{z-a_{2k}}{\epsilon_k})}{(a_{2N + 1} - z)^{1-\alpha}}dz\\
=&\sum_{k=0}^{N}\epsilon_k\int_{0}^{1} \frac{1-y}{(a_{2k}+\epsilon_k y)^{\alpha}(a_{2N + 1} - a_{2k}-\epsilon_k y)^{1-\alpha}}dy.
\end{align*}
We have $a_{2k}+\epsilon_k y\geq a_0=1$ and, using $\epsilon_k\leq 1$ and~\eqref{condN},
$a_{2N + 1} - a_{2k}-\epsilon_k y\geq a_{2N + 1}-a_{2N}-1= a_{2N}(a_1-1)-1\geq 1$.
It follows
$$
v_2(a_{2N+1})\leq \sum_{k=0}^{N}\epsilon_k\int_{0}^{1}dy=\sum_{k=0}^{\infty}\epsilon_k=\frac{\eta_\alpha}{4}.
$$
%[I do not suceed to prove the inequality above without~\eqref{condN}. You can suppress if there is
%an easy way to do it without the additional condition in~\eqref{condN}]
Similarly, we have that
\begin{align*}
 v_3(a_{2N+1})&\leq\sum_{k=0}^{N-1}\int_{a_{2k+1}-\epsilon_k}^{a_{2k+1}} \frac{z^{-\alpha}(1-\frac{a_{2k+1}-z}{\epsilon_k})}{(a_{2N + 1} - z)^{1-\alpha}}dz+\int_{a_{2N+1}-\epsilon_N}^{a_{2N+1}} \frac{z^{-\alpha}}{(a_{2N + 1} - z)^{1-\alpha}}dz\\
=&\sum_{k=0}^{N-1}\epsilon_k\int_{0}^{1} \frac{1-y}{(a_{2k+1}-\epsilon_k y)^{\alpha}(a_{2N + 1} - a_{2k+1}+\epsilon_k y)^{1-\alpha}}dy\\
&\hspace{7cm}+\int_{a_{2N+1}-\epsilon_N}^{a_{2N+1}} \frac{z^{-\alpha}}{a_{2N + 1}^{1-\alpha}(1 - z/a_{2N + 1})^{1-\alpha}}dz.\\
\end{align*}
To estimate the first term above
we notice first that $a_{2k+1}-\epsilon_k y-1\geq a_1-(1-b_\alpha^2)/4= (b_\alpha-1)(b_\alpha^2-b_\alpha-4)/(4b_\alpha)\geq 0$
since $0<b_\alpha <1$. Moreover, for $k\leq N-1$, $a_{2N + 1} - a_{2k+1}+\epsilon_k y\geq a_{2N + 1}-a_{2N - 1}\geq  a_{2N}(a_1-1)\geq 2$
by~\eqref{condN}.
To estimate the second term, we notice
$$
\int_{a_{2N+1}-\epsilon_N}^{a_{2N+1}} \frac{z^{-\alpha}}{a_{2N + 1}^{1-\alpha}(1 - z/a_{2N + 1})^{1-\alpha}}dz
=\int_{1-\epsilon_N/a_{2N+1}}^{1} \frac{1}{y^{\alpha}(1 - y)^{1-\alpha}}dz\leq \frac{\eta_\alpha}{4},
$$
using again~\eqref{condN}.
It follows
$$
 v_3(a_{2N+1})\leq \sum_{k=0}^{N-1}\epsilon_k\int_{0}^{1}dy+\frac{\eta_\alpha}{4}=\frac{\eta_\alpha}{2}.
$$
Therefore, $u(a_{2N + 2}) - u(a_{2N + 1})\leq-\eta_\alpha/4<0$, and this means that 
$$
\liminf_{t \to \infty} u(t) - \limsup_{t \to \infty} u(t) \leq -\eta_\alpha/4,
$$ 
from which $u$ does not have any limit at infinity.

\medskip

%{\br Now we discuss about the validity of our discussion on equation~\eqref{ode1} when the derivative is of Caputo type.} As it can be seen in Lemma 3.4 in~\cite{diethelm10}, Caputo and Riemann-Liouville derivative of a function $u$ agree if $u(0^+) = 0$ and both fractional derivatives exist. {\br E.T.: This holds if $u$ is an absolutely continuous function. This is not clear for me. However, I think that we can modify $f$ to have a Lipschitz function, fixing its discontinuities by linear functions whose total derivative is small...in that case we have enough regularity to replace R-L derivative by Caputo derivative in~\eqref{ode1}.} 

%%%%%%%%%%%%%%%%%%%%%%%%%%%%%%%%%%%%%%%%%%%%%%%
%%%%%%%%%%%%%%%%%%%%%%%%%%%%%%%%%%%%%%%%%%%%%%%
%%%%%%%%%%%%%%%%%%%%%%%%%%%%%%%%%%%%%%%%%%%%%%%
%%%%%%%%%%%%%%%%%%%%%%%%%%%%%%%%%%%%%%%%%%%%%%%
%%%%%%%%%%%%%%%%%%%%%%%%%%%%%%%%%%%%%%%%%%%%%%%
%%%%%%%%%%%%%%%%%%%%%%%%%%%%%%%%%%%%%%%%%%%%%%%
\section{Ergodic Large Time Behavior.}
\label{secLTB}

In this section we present some cases for which ergodic large time behavior~\eqref{asympt-exp} holds. 
The main assumption here follows the classical requirements of Namah and Roquejoffre~\cite{nr99}, see
Assumptions~\eqref{NR}. 

We have in mind the classical Eikonal case
\begin{eqnarray}\label{Heikonal}
&&  H(x,p)=F(x,p)-f(x)=a(x)|p|-f(x),
\end{eqnarray}
where $a,f:\T^N\to\R$ are Lipschitz continuous, $a(x)\geq \underline{a}>0$ and $f(x)\geq \mathop{\rm min}_{\T^N}f=0$. However, we are able to deal with Hamiltonians with superlinear growth on the gradient.

We notice that the convexity and the coercivity condition in~\eqref{H} leads to a quantitative growth for the Hamiltonian, that is
\begin{equation*}
H(x,p) \geq C^{-1} |p| - C, \qquad \text{for all $x\in\T^N$, $p\in\R^N$,}
\end{equation*}
for some constants $C > 0$. Since a similar condition is found when~\eqref{Hsuper} holds, throughout this section we assume the existence of a constants $C_H > 1$ and $m \geq 1$ such that
\begin{equation}\label{coerc-base}
F(x,p) \geq C_H^{-1} |p|^m - C_H, \qquad \text{for all $x\in\T^N$, $p\in\R^N$.}
\end{equation}

%But in general, the best quantitative coercivity property we can obtain using the
%coercivity in~\eqref{H} and the convexity in~\eqref{NR} is that there exists $C_H>0$ such that
%\begin{eqnarray}\label{coerc-base}
%&&  F(x,p)\geq \frac{1}{C_H}|p|-C_H,\qquad \text{for all $x\in\T^N$, $p\in\R^N$,}
%\end{eqnarray}

%As explained in the introduction, the proof of Namah-Roquejoffre Theorem~\cite[Theorem 1]{nr99} relies
%on three steps, the limiting one being the second one, i.e., to prove that
%$u(\cdot,t)+ct^\alpha$ converges on $Z$ when $t\to +\infty$. The main result
%of this section consists in proving this convergence under some additional assumptions,
%which are stated below. The large time behavior is an easy consequence, see Corollary~\ref{ltb-caputo}.

%The first set of assumptions is concerned with the behavior of the Hamiltonian
%$H(x,p)$ near $p=0$.
%The typical Hamiltonian one has in mind in this framework is the one of the classical Eikonal equation, namely
%\begin{eqnarray}\label{Heikonal}
%&&  H(x,p)=F(x,p)-f(x)=a(x)|p|-f(x),
%\end{eqnarray}
%where $a,f:\T^N\to\R$ are Lipschitz continuous, $a(x)\geq \underline{a}>0$ and $f(x)\geq \mathop{\rm min}_{\T^N}f=0$.
%In this case,~\eqref{H} and \eqref{NR} hold
%and
%\begin{eqnarray}\label{ineg-eikonal}
%F(x,p)\geq \underline{a}|p|, \qquad \text{for all $x\in\T^N$, $p\in\R^N$, $\underline{a}>0$.}
%\end{eqnarray}

We also require some assumption on the behavior of $F$ near $p = 0$ which is not reflected by~\eqref{coerc-base}.
In order to be able to deal with more general Hamiltonian, e.g., smooth ones
which are nonnegative and nondegenerate near $p=0$,
we introduce an additional assumption:
\begin{eqnarray}\label{Hsmooth0}
&& \text{There exists $\nu, r>0$ and $k\geq 1$ such that $F(x,p)\geq \nu |p|^k$ for all $x\in\T^N, p\in B(0,r)$.}
\end{eqnarray}
Thus, if~\eqref{H} or~\eqref{Hsuper} holds, then the later condition together with~\eqref{NR} lead to
\begin{eqnarray}\label{HsmoothR}
&& \text{for every $R>0$, there exists $\nu_R>0$ such that $F(x,p)\geq \nu_R |p|^k$ for all $p\in B(0,R)$.}
\end{eqnarray}
This is a sort of nondegeneracy condition in the sense that $F$ is not too flat around $p=0$.
%%%%
%\begin{remark}\label{remH} \ \\
%%(i) Notice that~\eqref{Hsmooth0} holds for~\eqref{Heikonal} with $k=1$ and $r=+\infty$.\\
%(i) Under this additional assumption~\eqref{Hsmooth0}
%together with~\eqref{H} and \eqref{NR}, we have
%\begin{eqnarray}\label{HsmoothR}
%&& \text{for every $R>0$, there exists $\nu_R>0$ such that $F(x,p)\geq \nu_R |p|^k$ for all $p\in B(0,R)$.}
%\end{eqnarray}
%(ii) A more natural assumption would be to assume~\eqref{Hsmooth0} for every $p\in\R^N$ (i.e., with $r=+\infty$).
%The problem is that the first line in~\eqref{H} does not hold anymore when $k>1$. It would require to prove
%more involved comparison and regularity results for~\eqref{eq}, which is beyond the scope of this paper.
%\end{remark}
%%%%

\medskip
By replacing $f$ with $f-\mathop{\rm min}_{\T^N}f$ in~\eqref{eq-new}, we may assume
without loss of generality that
\begin{eqnarray*}
	\mathop{\rm min}_{\T^N}f =0.
\end{eqnarray*}  
It follows that $c=0$ in~\eqref{ergodic}.

\medskip
As we will see later in the proof of Lemma~\ref{cvusurZ}, the solutions of~\eqref{eq}
for $x\in Z$
are strongly related to the solutions of the ODE $\partial_t^\alpha \E (t)+A|\E(t)|^k=0$,
for which we state a technical lemma.

%%%%%%%%%%%%%%%%%%%%%%%%
\begin{lema}\label{frac-ode}
For every $A>0$ and $k\geq 1$, there exists a unique positive
solution $\E\in C([0,\infty))\cap C^1((0,\infty))$ to
\begin{eqnarray}\label{edocaputo}
&& \partial_t^\alpha \E (t)+A |\E(t)|^k=0, \qquad \E(0)=1,
\end{eqnarray}
such that $\E(t) \searrow 0$ as $t\to +\infty$.
Moreover, there exists $C=C(A,k,\alpha)>0$ and, for all $\e>0$, there exists $C_\e=C_\e(A,k,\alpha)$
such that
\begin{eqnarray}\label{estimE}
\frac{C}{t^{\alpha/k}}\leq \E(t) \leq \frac{C_\e}{t^{\alpha/k-\e}}
\quad \text{for $t$ large enough.}  
\end{eqnarray}
\end{lema}
%%%%%%%%%%%%%%%%%%%%%%%

%%%%
\begin{remark}\label{rem-ode}
In the Eikonal case~\eqref{Heikonal}, the related ODE~\eqref{edocaputo}
reads $\partial_t^\alpha \E (t)+A\E(t)=0$ (with $k=1$) and
we have an explicit solution as recalled in~\eqref{mlf}, namely
the Mittag-Leffler function $E_\alpha(-At^\alpha)$. The
estimate~\eqref{estimE} is then optimal and given by~\eqref{estim-mlf}.
\end{remark}
%%%%

\begin{proof}[Proof of Lemma~\ref{frac-ode}]
Existence and uniqueness of the positive decreasing solution $\E$
of~\eqref{edocaputo} satisfying the lower bound
in~\eqref{estimE} is given by~\cite[Theorem 5.10]{fllx18}. 

Concerning estimates~\eqref{estimE}, by Remark~\ref{rem-ode} we can concentrate on the case $k > 1$.
Below, $c$ is a positive constant which may change line to line. Also, by $\partial^\alpha_t (1 + t)^{-p}$ we mean $\partial^\alpha_t (1 + \cdot)^{-p}(t)$.

\medskip

\noindent
\textit{Claim:} For each $p > 0$, there exists $c > 0$ just depending on $\alpha$ and $p$ such that
\begin{equation*}
-c t^{-\alpha} \leq \partial_t (1 + t)^{-p} \leq 0 \quad \mbox{for all} \ t > 1.
\end{equation*}

The upper bound is obvious. The lower bound can be obtained by a combination of~\cite[p.193]{diethelm10} and~\cite[Chapter 15, 7.3]{as64}, but we present here an alternative proof for completeness.

Using the definition of Caputo derivative, for $t > 1$ we see that
\begin{align*}
\partial^\alpha_t (1 + t)^{-p} & \geq \partial^\alpha_t[t/2, t] (1 + t)^{-p} + \partial^\alpha_t[0, t/2] (1 + t)^{-p}+  \int_{-\infty}^{0} \frac{ - 1}{|t - z|^{1 + \alpha}}dz \\
& \geq \partial^\alpha_t[t/2, t] (1 + t)^{-p} + \partial^\alpha_t[0, t/2] (1 + t)^{-p} - c t^{-\alpha},
\end{align*}
for some $c > 0$ just depending on $\alpha$. Using the Mean Value Theorem, there exists a constant $c$ depending on $p$ such that 
\begin{equation*}
\partial^\alpha_t[t/2, t] (1 + t)^{-p} \geq -c (1 + t)^{-p - 1} \int_{t/2}^{t} |t - z|^{- \alpha} dz \geq -c t^{-p - \alpha},
\end{equation*}
meanwhile, neglecting positive terms, we see that
\begin{align*}
\partial^\alpha_t[t/2, t] (1 + t)^{-p} \geq -\int_{0}^{t/2} \frac{(1 + z)^{-p}}{|t - z|^{1 + \alpha}} dz \geq -\int_{0}^{t/2} \frac{dz}{|t - z|^{1 + \alpha}} \geq -c t^{-\alpha},
\end{align*}
and joining the above inequalities we conclude the Claim.

%By~\cite[p.193]{diethelm10}, we have, for $p>0$,
%$\partial_t^\alpha (1+t)^{-p} =-\hat{C} t^{1-\alpha} \,\rule{0cm}{0cm}_2 F_1(1,1+p;2-\alpha;-t)$
%where $\rule{0cm}{0cm}_2 F_1$ is the Gauss’ hypergeometric function.
%By~\cite[Chapter 15, 7.3]{as64}, we obtain the asymptotic behavior
%of $\rule{0cm}{0cm}_2 F_1$  for large $t$, which yields
%$\partial_t^\alpha (1+t)^{-p}\sim -\hat{C} t^{-\alpha}$ as $t\to +\infty$. Notice that
%$\hat{C}$ depends on $p$ but the order is $\alpha$ for all $p>0$.
%It makes the asymptotic behavior of $\E(t)$ difficult to
%catch. 

Let $\epsilon > 0$ be small enough in order to have $p_\epsilon := \alpha/k - \epsilon > 0$. Take $C > 0$ large enough such that $C(1 + t)^{-p_\e} \geq C 2^{-p_\epsilon} \geq \E$ in $[0,1]$. By the Claim, it is possible to take $C$ larger if it is necessary to get
\begin{align*}
\partial_t^\alpha C(1+t)^{-\alpha/k +\e} + (C(1+t)^{-\alpha/k +\e})^k\geq 0 \quad \mbox{in} \ [1,+\infty).
\end{align*}
Then, by comparison, we arrive at $0\leq \E(t)\leq C_\e(1+t)^{-\alpha/k +\e}$
for all $t\geq 0$. This concludes the proof.
%
%Using this, since $\partial_t^\alpha (1+t)^{-\alpha/k +\e}$ is a continuous function on $[0,\infty)$
%which behaves like  $-\hat{C} t^{-\alpha}$ at infinity, for all $\e >0$,
%there exists $C_\e >0$ large enough 
%such that $\partial_t^\alpha C_\e(1+t)^{-\alpha/k +\e} + (C_\e(1+t)^{-\alpha/k +\e})^k\geq 0$
%for all $t>0$. By comparison, we obtain $0\leq \E(t)\leq C_\e(1+t)^{-\alpha/k +\e}$
%for all $t\geq 0$.
\end{proof}

%%%%%%%%%%%%%%%%%%%%%%%
%%%%%%%%%%%%%%%%%%%%%%%
%%%%%%%%%%%%%%%%%%%%%%%

In order to state our key result to obtain the large time behavior,
we need some definitions. Given two points $x_0, x_1 \in \R^N$, we denote $[x_0, x_1]$ the line segment joining $x_0$ and $x_1$. For a set of points $x_0, x_1,...,x_n$, with $n\in\N$, we denote
$$
[x_0, ..., x_n] = \bigcup \limits_{i=1}^{n} [x_i, x_{i-1}],
$$
that is, the polygonal curve joining the points $x_i, \ i=0,...,n$. The length of a finite polygonal line $[x_0,x_1,\cdots, x_n]$, $x_i\in\T^N$, 
is given by $\displaystyle \ell ([x_0,x_1,\cdots, x_n])= \sum_{i=1}^{n}|x_i-x_{i-1}|$.
A continuous curve $\gamma :[0,1]\to \T^N$ is said to be \textit{rectifiable}
if
\begin{eqnarray*}
\ell(\gamma):=\mathop{\rm sup} \limits_{\scriptsize\begin{array}{c}n\in\N\\ 0:=t_0<t_1<\cdots <t_n=1\end{array}}
\ell ([\gamma(t_0),\gamma(t_1),\cdots,\gamma(t_n)])<+\infty.
\end{eqnarray*}
We call $\ell(\gamma)$ the length of $\gamma$.

As explained in the introduction, the proof of Namah-Roquejoffre Theorem~\cite[Theorem 1]{nr99} relies
on three steps, the limiting one being the second one, i.e., to prove that
$u(\cdot,t)$ converges on $Z=\{f=0\}$ when $t\to +\infty$ (recall that we assume $c=0$).
We prove now this key result under the additional assumption
\begin{eqnarray}\label{ming}
&& {\rm argmin}\{g\}\cap Z\not=\emptyset.
\end{eqnarray}

The large time behavior is an easy consequence, see Corollary~\ref{ltb-caputo}.

%%%%%%%%
\begin{teo}\label{cvusurZ1}
Assume~\eqref{H} or~\eqref{Hsuper},~\eqref{NR},~\eqref{Hsmooth0} and~\eqref{ming}.
Assume further that for each $z \in Z$, there exists $x_0 \in Z \cap \mathrm{argmin} \{ g \}$ and a
rectifiable curve $\gamma: [0, 1] \to \T^N$  such that $\gamma(0) = x_0$, $\gamma(1) = z$ and $\gamma(t) \in Z$ for all $t \in [0,1]$.
Then, the unique solution $u$ to~\eqref{eq}-\eqref{u0} converges on $Z$, i.e.,
\begin{eqnarray*}
&& \text{for every  $x\in Z$,  $u(x,t)\to \min \{g\}$ as $t\to +\infty$.}
\end{eqnarray*}
\end{teo}
%%%%%%%%

Before giving the proof
of the theorem, we state the following key lemma.

%%%%%%%%
\begin{lema}\label{cvusurZ}
Assume hypotheses of Theorem~\ref{cvusurZ} hold.
%Assume~\eqref{H},~\eqref{NR},~\eqref{Hsmooth0} and~\eqref{ming}
Let $z \in Z$, $x_0 \in Z \cap \mathrm{argmin} \{ g \}$ and assume
that there exists a finite polygonal line $\gamma:=[x_0,x_1,\cdots, x_n\!\!:=z]$
lying in $Z$ and joining $x_0$ to $z$.
Then, the unique solution $u$ to~\eqref{eq}-\eqref{u0} satisfies
\begin{eqnarray}\label{ineg-finite}
&& \min \{g\} \leq u(z,t)\leq \min \{g\} + {\rm Lip}(g) \ell(\gamma)\E(t),
\end{eqnarray}
where $\E(t)\searrow 0$ as $t\to +\infty$ is a function which
depends on $H$, $||f||_\infty$, ${\rm Lip}(g)$, $N$ and $\ell(\gamma)$.
\end{lema}
%%%%%%%%

\begin{proof}[Proof of Lemma~\ref{cvusurZ}]
Without loss of generality we can assume $\min \{g\} = 0$. From this and~\eqref{NR}, we obtain that $0$ is a subsolution of~\eqref{eq}-\eqref{u0} in $Q$. Therefore, by comparison, $0\leq u$ in $Q$.

For the upper bound, the idea is to construct a function $U$ such that $U(z,t) \searrow 0$ as $t \to +\infty$ and such that $u(z,t) \leq U(z,t)$ for all $t > 0$. This is performed by an inductive procedure, building a sequence of functions $(U_i)_{0\leq i\leq n-1}$ which are supersolutions for the equation solved by $u$ but in the set $Q_i := Q \setminus \{ x_i \} \times [0,+\infty)$, with some control on the line $\{ x_i \} \times [0,+\infty)$ in order to use comparison principles for the Cauchy-Dirichlet problem.
%%of~\eqref{eq}-\eqref{u0} such that
%\begin{eqnarray}\label{majUi}
%U_{i-1}(x_i,t)\leq  {\rm Lip}(g)\sum_{j=1}^i |x_j-x_{j-1}|\E(t)\qquad \text{for all $t\geq 0$.}
%\end{eqnarray}
%Since, by comparison, we have
%\begin{eqnarray}\label{compuU}
%0\leq u\leq U_i \qquad\text{in $Q$},
%\end{eqnarray}
%the result~\eqref{ineg-finite} follows by choosing $i=n$ in~\eqref{majUi}.

We divide the proof in several steps.

\smallskip

\noindent{\it Step 1. Definition of $\E(t)$.}
By Lemma~\ref{frac-ode}, for every $A>0$ and $k\geq 1$, there exists a unique positive
solution $\E\in C([0,\infty))\cap C^1((0,\infty))$ to~\eqref{edocaputo}
such that $\E(t) \searrow 0$ as $t\to +\infty$.
Notice that, since $\E$ is nonincreasing, $\partial_t^\alpha \E (t)\leq 0$.
%Moreover, for all $\tilde{A}\geq A,$ we have
%$\partial_t^\alpha \E (t)+\tilde{A}|\E(t)|^k\geq 0$.

We now define $A$ and other constants, the definition of which
will be clear below.
We set $L := \mathrm{Lip}(g)$ and
\begin{eqnarray}\label{defM}
M:= L+ C_H^2 + C_H ||f||_\infty + LC_H(\sqrt{N} + \ell(\gamma)),
\end{eqnarray}
where $C_H$ appears in~\eqref{coerc-base} and
$\sqrt{N}={\rm diam}(\T^N)$.

From~\eqref{Hsmooth0} and~\eqref{HsmoothR}, we may define $\nu_{L+M}>0$ such that
\begin{eqnarray}\label{HsmoothRbis}
&& F(x,p)\geq \nu_{L+M} |p|^k  \quad \text{ for } |p|\leq L+M.
\end{eqnarray}
We then fix
\begin{eqnarray}\label{defA}
A=\frac{\nu_{L+M}L^{k-1}}{\sqrt{N} + \ell(\gamma)}
\end{eqnarray}
in~\eqref{edocaputo}.
Notice that we may assume without loss of generality that
that $A\leq 1$ by decreasing $\nu_{L+M}>0$ if necessary.
\smallskip

\medskip

\noindent{\it Step 2. Definition of the function $U_i$, $0\leq i\leq n-1$.}
We set
\begin{equation*}
U_0(x,t) = L |x - x_0|\E (t) + M d_{[x_0,x_1]}(x)
\end{equation*}
and
\begin{equation}\label{inductiveUi}
U_i(x,t) = L\sum_{j=1}^{i}|x_j-x_{j-1}|\E (t) +L|x - x_i| \E (t) + M d_{[x_i,x_{i+1}]}(x), \qquad 1\leq i\leq n-1,
\end{equation}
where $M$ is given by~\eqref{defM} and
$d_{[x_i,x_{i+1}]}$ denotes the (periodic) distance function to the segment $[x_i,x_{i+1}]$, that is, for each $x \in \T^N$ (cast as a point in $[0,1)^N$), we write
$$
d_{[x_i, x_{i + 1}]}(x) = \inf_{y \in [x_i, x_{i + 1}], k \in \Z^N} |x + k - y|.
$$

This is a $1$-Lipschitz continuous. At the points where it is differentiable we have the gradient meets
$\widehat{x - p_i}$, where $p_i$ is the projection of $x$ to the segment, from which $|D d_{[x_i,x_{i+1}]}(x)|=1$. In addition, for each point in the set of non differentiability of $d_{[x_i, x_{i+1}]}$ which do not lie on the segment $[x_i, x_{i+1}]$, there is not $C^1$ function touching the function from below.

\medskip
\noindent{\it Step 3. The initial supersolution $U_0$.} We prove actually that $U_0$
is a supersolution.
At first, for $t=0$ and all $x\in\T^N$, since $L= \mathrm{Lip}(g)$ and $g(x_0)=0$, we have 
$$
U_0(x,0)\geq L|x-x_0|\geq g(x)-g(x_0)=g(x).
$$
If $t>0$, $x\not\in [x_0,x_1]$ and $U_0$ is $C^1$ at $(x,t)$, by the choice of the constant $M_0$ we use the coercivity properties of the Hamiltonian to write the following computation
holds
\begin{align*}
\partial_t^\alpha U_0+F(x,DU_0)&-f(x)\\
&=
L|x-x_0|\partial_t^\alpha \E (t) +F\left( x, L\E(t)\widehat{x-x_0}+M_0 Dd_{[x_0,x_1]}(x)\right)-f(x)\\
&\geq
L\sqrt{N}\partial_t^\alpha \E (t) + \frac{1}{C_H}\left| L\E(t)\widehat{x-x_0}+M_0 Dd_{[x_0,x_1]}(x)\right| -C_H-||f||_\infty\\
&\geq
\frac{M}{C_H}-\frac{L}{C_H}-C_H-||f||_\infty
+L\sqrt{N}\left(\partial_t^\alpha \E (t) +A|\E(t)|^k\right) -L\sqrt{N}A|\E(t)|^k\\
&\geq
\frac{M}{C_H}-\frac{L}{C_H}-C_H-||f||_\infty -L\sqrt{N}A,
\end{align*}
where we used~\eqref{coerc-base}, $0\leq \E(t)\leq 1$, $\partial_t^\alpha \E (t)\leq 0$,~\eqref{edocaputo}
and $\sqrt{N}={\rm diam}\,\T^N$.
By the choice of $M$ in~\eqref{defM} and since $A\leq 1$,
the right hand side of the previous inequality is nonnegative. So $U_0$ is a supersolution outside~$[x_0,x_1]$.

Now, let  $t>0$, $x\in [x_0,x_1]$ and consider any $C^1$ test-function $\varphi$ touching $U_0$ from below at $(x,t)$, i.e.,
$U_0(y,s)\geq \varphi(y,s)$ and $U_0(x,t)=\varphi(x,t)$. 
Since $U_0$ is sufficiently smooth in time, the following computation holds in a classical way,
\begin{align*}
\partial_t^\alpha \varphi (x,t)
= 
\int_{-\infty}^t \frac{\varphi(x,t)-\varphi(x,s)}{|t-s|^{1+\alpha}}ds
&\geq
\int_{-\infty}^t \frac{U_0(x,t)-U_0(x,s)}{|t-s|^{1+\alpha}}ds\\
&=\partial_t^\alpha U_0 (x,t)
=L|x-x_0|\partial_t^\alpha \E (t).
\end{align*}
Since the right hand side of the above inequality is zero at $x=x_0$ and
$F(x,p)\geq 0$, the supersolution property holds at $x=x_0$. 

It remains to prove that $U_0$ is a supersolution for
$x\in [x_0,x_1]\setminus \{x_0\}$. In this case, $x-h \widehat{x-x_0}\in [x_0,x_1]$ for $h>0$
enough small and $d_{[x_0,x_1]}(x)=d_{[x_0,x_1]}(x-h \widehat{x-x_0})=0$. It follows
\begin{align*}
\varphi (x-h \widehat{x-x_0},t)-\varphi (x,t)
&\leq
U_0 (x-h \widehat{x-x_0},t)-U_0 (x,t)\\
&= L|x-h \widehat{x-x_0}-x_0|\E(t)-L|x-x_0|\E(t)=-h\E(t).
\end{align*}
Dividing by $h$ and letting $h\searrow 0$ we obtain a lower-bound for $|D\varphi(x,t)|$,
\begin{eqnarray}\label{borneinf}
&& |D\varphi(x,t)|\geq \langle  D\varphi(x,t), \widehat{x-x_0}\rangle \geq L\E(t).
\end{eqnarray}
Noticing that $U_0$ is Lipschitz continuous with constant $L+M$ in space, we have
also an upper-bound $|D\varphi(x,t)|\leq L+M$.
For $x\in [x_0,x_1]\setminus \{x_0\}$, recalling $f(x)=0$, we use the behavior of $F$ near the origin in~\eqref{HsmoothR} together with
~\eqref{borneinf} to write
\begin{align*}
\partial_t^\alpha \varphi (x,t)+F(x,D\varphi(x,t))&-f(x)\\
&\geq
\partial_t^\alpha U_0 (x,t)+ \nu_{L+M}|D\varphi(x,t)|^k\\
&=
L|x-x_0|\partial_t^\alpha \E(t)+ \nu_{L+M}L^k|\E(t)|^k\\
&\geq 
L\sqrt{N}\left( \partial_t^\alpha \E(t) + A|\E(t)|^k \right)-L\sqrt{N}A|\E(t)|^k
+\nu_{L+M}L^{k}|\E(t)|^k\\
&\geq
L \left( \nu_{L+M}L^{k-1}-A\sqrt{N}\right)|\E(t)|^k
\end{align*}
by~\eqref{edocaputo}.
Recalling the choices of $A$ in~\eqref{defA},
the right hand side of the above inequality is nonnegative. 

Since the viscosity inequality follows at once in the points where the distance function cannot be touched from below, the previous discussion leads to conclude that
$U_0$ is a supersolution of~\eqref{eq}-\eqref{u0} in $Q$.
By comparison, we obtain $u(x,t) \leq U_0(x,t)$ for all $(x,t) \in \bar Q$, from which, in particular we get that
\begin{equation*}
u(x_1, t) \leq L |x_1 - x_0| \E(t) \quad \mbox{for all} \ t.
\end{equation*}

\medskip

\noindent{\it Step 4. Proof by induction for $U_i$, $i\geq 1$.} %If we assume the convention that definition~\eqref{inductiveUi} for $i = 0$ is performed defining as zero the sum from $1$ to $0$, then this definition meets $U_0$ of the previous step.
By induction, we will prove that $U_i$ in~\eqref{inductiveUi} satisfies
\begin{align}\label{compuU}
\partial_i^\alpha U_i + H(x, DU_i) \geq 0 \quad \mbox{in} \ Q_i, \qquad u \leq U_{i} \quad \mbox{in} \ \partial_p Q_i,
\end{align}
where $\partial_p Q_i$ is the parabolic boundary $\{ x_i \} \times (0, t) \cup \T^N \times \{ 0 \}$.

We first deal with the Cauchy-Dirichlet condition. By assumption, we have
$0\leq u(x_i,t)\leq U_{i-1}(x_i,t)$ for every $t\geq 0$. When $i=1$, it follows
\begin{eqnarray*}
&& u(x_1,t)\leq U_0(x_1,t)=L|x_1-x_0|\E(t)\leq \sqrt{N}L\E(t)\leq U_1(x_1,t). 
\end{eqnarray*}  
When $i\geq 2$, we have
\begin{eqnarray*}
&& u(x_i,t)\leq U_{i-1}(x_i,t)=L\sum_{j=1}^{i-1}|x_j-x_{j-1}| \E(t)
+ L|x_i-x_{i-1}|\E(t)=U_i(x_i,t). 
\end{eqnarray*} 
For $t=0$ and all $x\in\T^N$, we have, using the triangle inequality,
$$
U_i(x,0) =L\sum_{j=1}^{i}|x_j-x_{j-1}| \E(0)+ L|x-x_i|\E(0)
\geq L|x-x_0|\geq g(x)-g(x_0)=g(x),
$$
so the initial condition is satisfied.

\medskip

Now we deal with the PDE in~\eqref{compuU}. 
%It is therefore enough to prove that $U_i$ is a supersolution of~\eqref{eq}-\eqref{u0} in
%$Q_i:=Q\setminus (\{x_i\}\times [0,+\infty))$
%to recover~\eqref{compuU} in $Q$ by applying the Cauchy-Dirichlet comparison in
%$Q_i$. {\br REFERENCES OR COMMENTS.}
The proof follows the same
lines as the one in Step 4, so we only sketch it.

If $t>0$ and $x\not\in [x_i,x_{i+1}]$, then $U_i$ is $C^1$ at $(x,t)$ and we can
do the same computation as in Step 4 to obtain
\begin{align*}
 \partial_t&^\alpha U_i+F(x,DU_i)-f(x)\\
&= 
L\ell(\gamma)\partial_t^\alpha \E(t)
+ L |x-x_i|\partial_t^\alpha \E(t)+F\left(x,L\widehat{x-x_i} \E(t)+M Dd_{[x_i,x_{i+1}](x)}\right)-f(x)\\
&\geq
L(\ell(\gamma)+\sqrt{N})\partial_t^\alpha \E(t)
+\frac{M-L}{C_H}-C_H-||f||_\infty\\
&\geq
L(\ell(\gamma)+\sqrt{N})\left( \partial_t^\alpha \E(t) +A |\E(t)|^k\right)
+ \frac{M}{C_H}- \frac{L}{C_H}-C_H-||f||_\infty- L(\ell(\gamma)+\sqrt{N})A|\E(t)|^k.
\end{align*}

By the choice of $M$ in~\eqref{defM}, recalling that $A\leq 1$,
we obtain that the right hand side of the above inequality is nonnegative.

Now, let  $t>0$, $x\in [x_i,x_{i+1}]\setminus\{x_i\}$
and consider any $C^1$ test-function $\varphi$ touching $U_i$ from below at $(x,t)$, i.e.,
$U_i(y,s)\geq \varphi(y,s)$ and $U_i(x,t)=\varphi(x,t)$. 
As in Step 4, we check easily that $\partial_t^\alpha \varphi (x,t)\geq \partial_t^\alpha U_i (x,t)$
and, since $d_{[x_i,x_{i+1}]}(x-h \widehat{x-x_i})=0$ for $h>0$ small enough, that
$|D\varphi(x,t)|\geq L\E(t)$.
Moreover, $|D\varphi(x,t)|\leq L+M$.
It follows
\begin{align*}
\partial_t^\alpha \varphi (x,t)&+F(x,D\varphi(x,t))-f(x)\\
&\geq
L(\ell(\gamma)+\sqrt{N})  
\left( \partial_t^\alpha \E(t) +A |\E(t)|^k\right)
- L(\ell(\gamma)+\sqrt{N})A|\E(t)|^k+  \nu_{L+M}L^k |\E(t)|^k\\
&\geq
L \left(\nu_{L+M}L^{k-1}- (\ell(\gamma)+\sqrt{N})A\right)|\E(t)|^k
\end{align*}
recalling that $f=0$ on $[x_i,x_{i+1}]$.
By the choice of $A$ in~\eqref{defA},
the right hand side of the above inequality is nonnegative.

It ends the proof of the supersolution property for $U_i$, and therefore the inductive process.

Then, using Cauchy-Dirichlet comparison principles in~\cite{namba18} (or standard adaptation of the comparison principles presented in~\cite{ty17} when the Dirichlet condition is satisfied pointwisely), we conclude that $u \leq U_i$ in $Q_i$, from which we get $u(z,t) \leq U_i(z,t)$ for all $t > 0$. This concludes the proof.
\end{proof}
%%%%%%%%%%%%%%%%%%%%%%%

%%%%%%%%%%%%%%%%%%%%%%%%%%%%%
\begin{proof}[Proof of Theorem~\ref{cvusurZ1}]
We may assume without loss of generality that ${\rm min}\, g=0$. It follows
that $0\leq u$ in $Q$ and we need to prove that $u(x,t)\to 0$ on $Z$
as $t\to +\infty$.

Let $\e>0$ and $f_\e:= (f-C_f \e)_+$, where $C_f:={\rm Lip}(f)$. Then
$f_\e$ is a periodic function that satisfies the following properties,
\begin{eqnarray*}
&& 0\leq f_\e\leq f, \qquad ||f_\e - f||_\infty \leq C_f\e,\\
&& Z=\{ f=0\}\subset \{d_Z(x)\leq \e\}\subset Z_\e=\{f_\e=0\}.  
\end{eqnarray*}  
To prove the last inclusion, let $x\in \T^N$ such that $d_Z(x)\leq\e$.
Then there exists $x_Z\in Z$ such that $|x-x_Z|\leq \e$ and
$0\leq f(x)\leq f(x_Z)+C_f|x-x_Z|\leq C_f\e$. Thus $f_\e(x)=0$.

Now we consider~\eqref{eq-new}-\eqref{u0} by replacing $f$ with $f_\e$.
Notice that the assumptions of Theorem~\ref{cvusurZ1} and Lemma~\ref{cvusurZ}
still holds. In particular, there exists a unique solution $u_\e$.
Moreover, $u_\e\pm ||f-f_\e|| c_{\alpha,\alpha}^{-1}t^\alpha$, where $c_{\alpha,\alpha}$
appears in~\eqref{cst-deriv},
are respectively a super- and a subsolution of~\eqref{eq-new}-\eqref{u0} with $f$.
By comparison, we get
\begin{eqnarray}\label{uueps}
&& ||u-u_\e||\leq \frac{C_f}{c_{\alpha,\alpha}} \e t^\alpha. 
\end{eqnarray}   
Let $z\in Z$ and $\gamma:[0,1]\to Z$ be a rectifiable polygonal
curve such that $\gamma(0)=x_0\in Z \cap \mathrm{argmin} \{ g \}$ and $\gamma(1)=z$.
By assumption, there exists a sequence of subdivision $t_0^k:=0 < t_1^k <\cdots < t_{n_k}^k:=1$,
$k\in\N$ and a finite polygonal line $\gamma_k:=[\gamma(t_0^k),\cdots,\gamma(t_{n_k}^k)]$
which satisfies $\ell (\gamma_k) \nearrow \ell(\gamma)$ as $k\to\infty$.

In particular, we can prove that ${\rm dist}(\gamma_k, \gamma)\to 0$ as $k\to\infty$.
It follows that there exists $k_\e\in\N$ such that $\gamma_{k_\e}\subset Z_\e$
and $\gamma_{k_\e}$ is a finite polygonal line joining $x_0$ and $z$.
We can apply Lemma~\ref{cvusurZ} to obtain
\begin{eqnarray}\label{estimueps}
0\leq u_\e(z,t)\leq {\rm Lip}(g)\ell (\gamma_{k_\e})\E(t), \qquad \text{for all $t\geq 0$.} 
\end{eqnarray}
A priori, $\E$ depends on $\gamma_{k_\e}$ through the dependence of $A$ with respect to
$\ell(\gamma_{k_\e})$ in~\eqref{edocaputo} but, since $\ell (\gamma_k)\leq \ell(\gamma)<+\infty$,
we can fix $A$ and $\E$ independently of $\e$. From~\eqref{uueps}, we finally
obtain
\begin{eqnarray*}
0\leq u(z,t)\leq \frac{C_f}{c_{\alpha,\alpha}} \e t^\alpha+{\rm Lip}(g)\ell (\gamma)\E(t), \qquad \text{for all $t\geq 0$.} 
\end{eqnarray*}
Sending first $\e\to 0$ and then $t\to +\infty$, we conclude.
\end{proof}
%%%%%%%%%%%%%%%%%%%%%%%

In the Eikonal case (c.f.~\eqref{Heikonal}), we have an explicit formula for $\E$ in Lemma~\ref{cvusurZ},
see Remark~\ref{rem-ode}, which allows to deal with
more involved $Z$. More precisely, we consider subsets $Z$ satisfying
the following assumption
\begin{eqnarray}
\label{hyp-box-count}
&&\begin{array}{c}
\text{There exists $D\geq 1$ and $C>0$ such that, for all $\e>0$ and $x\in Z(\e):=\{d_Z<\e\}$,}\\
\text{there exists $x_0\in Z \cap \mathrm{argmin} \{ g \}$ and a finite polygonal line $\gamma_\e\subset Z(\e)$}\\
\text{such that $\gamma_\e$ is formed by at most $C\e^{-D}$ lines of length $\e$.}
\end{array}  
\end{eqnarray}

We remark that a set $Z$ satisfying Assumption~\eqref{hyp-box-count} is a curve of box-counting dimension $D$,
see Falconer~\cite{falconer90}. In several interesting cases, box-counting dimension agrees with Hausdorff dimension
and when $D > 1$ then the curve have infinite length.

%%%%%%%%%%%%
\begin{teo} \label{cvZ-eikonal} (Eikonal case)
Assume~\eqref{H},~\eqref{NR},~\eqref{HsmoothR} with $k = 1$, and~\eqref{hyp-box-count} for $1\leq D <\frac{3}{2}$.
Then, the unique solution $u$ to~\eqref{eq}-\eqref{u0} converges on $Z$, i.e.,
\begin{eqnarray*}
&& \text{for every  $x\in Z$,  $u(x,t)\to {\rm min}\, \{ g \}$ as $t\to +\infty$.}
\end{eqnarray*}
\end{teo}
%%%%%%%%%%%%

%%%%%%%%%%%%%%%%%%%%%%%%%%%%%
\begin{proof}[Proof of Theorem~\ref{cvZ-eikonal}]
We proceed exactly as in the proof of Theorem~\ref{cvusurZ1}, where $\gamma_{k_\e}$
is replacing with $\gamma_\e$ given by Assumption~\eqref{hyp-box-count}.
From~\eqref{uueps} and~\eqref{estimueps}, we arrive at
\begin{eqnarray}\label{estim945}
&& 0\leq u(z,t)\leq \frac{C_f}{c_{\alpha,\alpha}} \e t^\alpha
+ {\rm Lip}(g)\ell (\gamma_{\e})\E(t), \qquad \text{for all $t\geq 0$.} 
\end{eqnarray}
The difference with the proof of Theorem~\ref{cvusurZ1} is that $\gamma$
is not necessarily rectifiable anymore. But we can estimate the length
of $\gamma_\e$ thanks
to~\eqref{hyp-box-count}
and take profit of the explicit formula for the solution
of~\eqref{edocaputo} in the Eikonal case $k=1$.

The constant $C>0$ below may change line to line but it does not depend
neither on $\e$ nor on $t$.
We have
\begin{eqnarray*}
&& \ell (\gamma_{\e})\leq C\e^{-D}\e, \qquad\text{from Assumption~\eqref{hyp-box-count},}\\   
&& \E(t)=E_\alpha(-At^\alpha)\leq \frac{C}{At^\alpha}, \qquad\text{from Remark~\ref{rem-ode} and~\eqref{estim-mlf}},\\
&& A\geq \frac{1}{C \ell (\gamma_{\e})}, \qquad\text{from~\eqref{defA} and Assumption~\eqref{HsmoothR} with $k = 1$}.
\end{eqnarray*}
Plugging all the estimates in~\eqref{estim945}, we end with
\begin{eqnarray}\label{estim520}
&& 0\leq u(z,t)\leq  C\left(\e t^\alpha+\frac{1}{\e^{2(D-1)}t^\alpha}\right), \quad\text{for all $t\geq 0$, $\e>0$.}
\end{eqnarray}
Minimizing over $\e>0$, we obtain 
\begin{eqnarray*}
&& 0\leq u(z,t)\leq  C t^{\alpha \frac{2D-3}{2D-1}}, 
\end{eqnarray*}
and the right hand side tends to 0 as $t\to +\infty$ when $D<\frac{3}{2}$.
\end{proof}
%%%%%%%%%%%%%%%%%%%%%%%

%%%%
\begin{remark}\label{classique-eikonal}
Notice that the condition on $D$ does not depend on $\alpha\in (0,1)$.
If we use the same approach in the classical case ($\alpha=1$), then we can repeat
the above proof with $\E(t)=e^{-At}$ and the right hand side of~\eqref{estim520} now reads
$C\e t + \e^{1-D} e^{-t\e^{D-1}/C}$. We can prove that the minimum over $\e>0$ tends to 0
as $t\to +\infty$ if and only if $D<2$. Even, if we obtain a more general result than in the fractional
case $\alpha\in (0,1)$, this result is not optimal since the convergence on $Z$ holds
for any $Z$ and ${\rm argmin}\{g\}$ even without
any connectedness requirement (see Introduction).
\end{remark}
%%%%

%%%%%%%%
\begin{cor}\label{ltb-caputo}
Under the assumptions of Theorem~\ref{cvusurZ1} or~\ref{cvZ-eikonal},
the unique solution $u$ of~\eqref{eq}-\eqref{u0} satisfies
\begin{eqnarray*}
u(x,t)+ct^\alpha -v(x) \to 0 \quad\text{uniformly as $t\to +\infty$,}
\end{eqnarray*}
where $c= -\mathop{\rm min}_{\T^N}f$ and $v$ is the unique solution of~\eqref{ergodic}
satisfying $v={\rm min}_{\T^N}g$ on $Z$.
\end{cor}
%%%%%%%

The proof of the corollary follows the procedure
described in the introduction.  We only sktech it.
By comparison, we obtain that $u+ct^\alpha$ is uniformly bounded
in $\T^N\times [0,+\infty)$. We then apply Theorems~\ref{teoHoldert} and~\ref{holderx}  
to  $u+ct^\alpha$ to prove Step 1. Step 2 follows from Theorem~\ref{cvusurZ1} or~\ref{cvZ-eikonal}
and Step 3 is classical. For details, we refer the reader
to the survey of Barles in~\cite{abil13} or Namah-Roquejoffre~\cite{nr99}.
\bigskip

\noindent {\bf Acknowledgements.} 
Part of this work was done during a visit of E.T. and M.Y. to the Institut de Recherche Math\'ematique de Rennes. They acknowledge the hospitality of the Institut.
O.L. is partially supported by the Agence Nationale de la Recherche
(MFG project ANR-16-CE40-0015-01 and Centre Henri Lebesgue ANR-11-LABX-0020-01).
E.T. was partially supported by  Conicyt PIA Grant No. 79150056, Foncecyt Iniciaci\'on No. 11160817, and Dicyt - Apoyo Asistencia a Eventos 2018. M.Y. is partially supported by Escuela Polit\'ecnica Nacional, Proyecto PII-DM-2019-01.

%%%%%%%%%%%%%%%%%%%%%%%%%%%%%%%%%%%% 
%%% BIBLIO AVEC BIBTEX
%\bibliographystyle{plain} 
%\bibliography{../../biblio} 
%\end{document}

\end{document}